\documentclass[10pt,reqno]{amsart} 

\usepackage{amssymb,latexsym,  
amscd, 
amsthm, 
graphicx, epsfig
}
\usepackage{cite} 

\usepackage[height=190mm,width=130mm]{geometry} 

\theoremstyle{plain}

\newtheorem{lemma}{Lemma}
\newtheorem{corollary}{Corollary}
\newtheorem{proposition}{Proposition}

\theoremstyle{definition}
\newtheorem{definition}{Definition}

\theoremstyle{remark}
\newtheorem{remark}{Remark}

\newcommand{\Z}{\mathbb{Z}}
\newcommand{\W}{\mathcal{W}} 
\newcommand{\U}{\mathcal{U}}

\newcommand{\Oo}{\mathcal{O}}
\newcommand{\C}{\mathbb{C}}

\numberwithin{equation}{section} 

\newcommand{\nn}{\nonumber \\}

 \newcommand{\res}{\mbox{\rm Res}}

\renewcommand{\hom}{\mbox{\rm Hom}}

\newcommand{\wt}{\mbox{\rm wt}\ }

\newcommand{\N}{\mathbb{N}}

\newcommand{\F}{\mathcal{F}}

\newcommand{\V}{\mathcal{V}}
\newcommand{\one}{\mathbf{1}}

\begin{document}

\title[Characterization of codimension one foliations  
 on complex curves]   
{Characterization of codimension one foliations on complex curves by connections} 
\author{A. Zuevsky}
\address{Institute of Mathematics \\ Czech Academy of Sciences\\ Praha}

\email{zuevsky@yahoo.com}






\begin{abstract}
A way to characterize the space of leaves 
 of a foliation in terms of connections is proposed.  
A particular example of vertex algebra cohomology of 
codimension one foliations on complex curves is considered. 
Mupltiple applications in Mathematical physics are revealed. 
\end{abstract}

\vskip12pt  

\maketitle
\section{Introduction}
\label{introduction}
The theory of foliations involves a bunch of approaches \cite{Bott, BS, BH, CM, 
 Fuks1, Fuks2, GF1, GF2, GF3, Losik} and many others. 
In certain cases it is useful to express cohomology in terms of connections and 
use the language of connections in order to study leave spaces of foliations.   
Connections appear in conformal field theory \cite{FMS, BZF} in definitions of many notions and formulas. 
Vertex algebras, generalizations of ordinary Lie algebras, are essential in conformal field theory. 
The theory of vertex algebra characters is a rapidly developing 
field of studies. Algebraic nature of methods applied in this field helps to understand and compute  
the structure of vertex algebra characters. 
On the other hand, the geometric side of vertex algebra characters is in associating their formal parameters with 
local coordinates on a complex variety.   
Depending on geometry, one can obtain various consequences for a vertex algebra and its space 
of characters, and vice-versa, one can study geometrical property of a manifold by using 
algebraic nature of a vertex algebra attached. 
In this paper we use the vertex algebra cohomology theory \cite{Huang} for characterization   
codimension one foliations on smooth complex curves \cite{Gu1, Gu2}.
The arbitrary condimension case will be considered in elsewhere.  

The general theory of characterizations of leaves of codimension one foliations  
 on complex curves has multiples applications in Mathematical Physiscs.  
In particular, since the notion connections used to define cohomology stems from 
reduction formulas \cite{Zhu} for correlation functions for vertex algebras with formal parameters 
identified with local parameters on Riemann surfaces, the codimension one foliation formulas 
are useful in studies of spaces of correlation functions. 
The double complexes used in this paper to define a cohomology originate from 
consideration of cohomology of vertex algebras in \cite{Huang, Hu3}.
Recall that vertex algebras play central role in many aspects of Mathematical Physics. 
The cohomology of vertex algebras is an important questions not only in the theory of vertex 
algebras but also in applications to Conformal Field Theory. 
 
The plan of the paper is the following. 
In Section \ref{connections} we describe the approach to cohomology in terms of connections. 
In Section \ref{spaces} we define the spaces for double complex associated to a quasi-conformal
 grading-restricted vertex algebra. Non-vanishing and canonicity of the construction is proved. 
In Section \ref{coboundary} coboundary operators are defined. It is shown that combining with the double complex 
spaces they determine chain-cochain double complex. 
In Section \ref{firstcohomology} we determine the first 
vertex algebra cohomologies of a codimension one foliation. 
Corresponding cohomological classes are considered in Section \ref{cohomological}. 
In Section \ref{applications} we consider multiple applications in Mathematical Physics. 
In Appendixes we provide the material needed for construction of the vertex algebra cohomology 
of foliations.
In Appendix \ref{grading} we recall the notion of a quasi-conformal grading-restricted vertex algebra 
and its modules.
 In Section \ref{valued} the space of $\W$-valued rational sections of a vertex algebra bundle is described.  
In Section \ref{properties} properties of matrix elements for elements of the space $\W$ are listed. 
In Section \ref{composable} maps composable with a number of vertex operators are defined. 
Appendix \ref{proof} contains proofs of Lemmas \ref{empty}, \ref{nezu}, \ref{subset}, and Proposition \ref{nezc}. 
\section{Cohomology in terms of connections}
\label{connections}
In various situations it is sometimes effective to use an interpretation of cohomology in terms of 
connections. 
In particular in our supporting example of vertex algebra cohomology of codimension one foliations. 
It is convenient to introduce multi-point connections over a graded space  
and to express coboundary operators and cohomology in terms of connections: 
\[
\delta^n \phi  \in G^{n+1}(\phi),  
\]
\[
 \delta^n \phi= G(\phi). 
\]
Then the cohomology is defined as the factor space 
\[
H^n= {\mathcal Con}^n_{cl; \; }/G^{n-1},   
\]
of closed multi-point connections with respect to the space of connection forms defined below. 
\subsection{Multi-point holomorphic connections} 
We start this section with definitions of holomorphic multi-point connections on a smooth complex variety.  
Let $\mathcal X$ be a smooth complex variety and $\V \to \mathcal X$ a holomorphic vector bundle over $\mathcal X$. 
Let $E$ be the sheaf of holomorphic sections of $\V$.
 Denote by $\Omega$
 the sheaf of differentials on $\mathcal X$. 
A holomorphic connection $\nabla$ on $E$ is a $\C$-linear map 
\[
\nabla: E \to E \otimes \Omega, 
\]
 satisfying the Leibniz formula
\[
\nabla(f\phi)= \nabla f \phi + \phi \otimes dz, 
\]
for any holomorphic function $f$.  
Motivated by the definition of the holomorphic connection $\nabla$ 
 defined for a vertex algebra bundle (cf. Section 6, \cite{BZF}) over  
a smooth complex variety $\mathcal X$, we introduce the definition of 
the multiple point holomorphic connection over $\mathcal X$. 
\begin{definition}
\label{mpconnection}
Let $\V$ be a holomorphic vector bundle over $\mathcal X$, and $\mathcal X_0$ its subvariety.  
A holomorphic multi-point 
connection $\mathcal G$ on $\V$    
is a $\C$-multi-linear map 
\[
\mathcal G: E  \to  E \otimes \Omega,
\]  
such that for any holomorphic function $f$, and two sections $\phi(p)$ and $\psi(p')$ at points  
$p$ and $p'$ on $\mathcal X_0$ correspondingly, we have 
\begin{equation}
\label{locus}
\sum\limits_{q, q' \mathcal X_0 \subset \mathcal X} \mathcal G\left( f(\psi(q)).\phi(q') \right) = f(\psi(p')) \; 
\mathcal G \left(\phi(p) \right) + f(\phi(p)) \; \mathcal G\left(\psi(p') \right),    
\end{equation}
where the summation on left hand side is performed over a locus of points $q$, $q'$ on $\mathcal X_0$. 
We denote by ${\mathcal Con}_{\mathcal X_0}(\mathcal S)$ the space of such connections defined over a
smooth complex variety $\mathcal X$. 
We will call $\mathcal G$ satisfying \eqref{locus}, a closed connection, and denote the space of such connections 
by ${\mathcal Con}^n_{\mathcal X_0; cl}$.  
\end{definition}
Geometrically, for a vector bundle $\V$ defined over a complex variety $\mathcal X$,
a multi-point holomorphic connection \eqref{locus} relates two sections $\phi$ and $\psi$ of $E$ at points $p$ and $p'$
with a number of sections at a subvariety $\mathcal X_0$ of $\mathcal X$.  
\begin{definition}
We call
\begin{equation}
\label{gform}
G(\phi, \psi) = f(\phi(p)) \; \mathcal G\left(\psi(p') \right)  + f(\psi(p')) \; \mathcal G \left(\phi(p) \right)   
- \sum\limits_{q, q' \mathcal X_0 \subset \mathcal X} \mathcal G\left( f(\psi(q')).\phi(q) \right), 
\end{equation}
the form of a holomorphic connection $\mathcal G$. 
The space of form for $n$-point holomorphic connection forms will be denoted by $G^n(p, p', q, q')$.  
\end{definition}
\begin{definition}
\label{fpconnection}
A fixed point holomorphic connection on $E$ is defined by the condition 
\begin{equation}
\label{locusfixed}
\sum\limits_{p_0; \; q, q' \in \mathcal X_0 \subset \mathcal X } \mathcal G\left( f(\psi(q')).\phi(q) \right) = f(\psi 
(p'_0)) \; 
\mathcal G \left(\phi(p) \right) + f(\phi(p)) \; \mathcal G \left(\psi(p_0) \right),    
\end{equation}
where a point $p_0 $ is fixed on $\mathcal X_0$.  
\end{definition}
\begin{definition}
\label{onepconnection}
A holomorphic connection defined for a vector bundle $\V$ over a smooth complex variety $\mathcal X$ 
(the 
two point case of the multi-point holomorphic connection \eqref{locus}) is called a two 
 point connection
 when for any holomorphic function $f$, 
\begin{equation} 
\label{onepointlocus}
\mathcal G\left( f(\psi(p') \right). \phi(p)) = f(\psi(p'))\; \mathcal G(\phi(p)) + f(\phi(p))\; \mathcal G(\psi(p')),  
\end{equation}
for two sections $\psi(p')$ and $\phi(p)$ of $E$.   
We denote the space of such connections as ${\mathcal Con}^2_{p, p_0; \; \mathcal X_0}$. 
\end{definition}
Let us formulate another definition which we use in the next section: 
\begin{definition}
\label{transcon}
We call a multi-point holomorphic connection $\mathcal G$ 
 the transversal connection, i.e., when it satisfies 
\begin{equation}
\label{transa0}
f(\psi(p'))\; \mathcal G(\phi(p)) + f(\phi(p)) \; \mathcal G(\psi(p'))=0. 
\end{equation}
We call 
\begin{equation}
\label{transa}
G_{tr}(p, p')=(\psi(p'))\; \mathcal G(\phi(p)) + f(\phi(p)) \; \mathcal G(\psi(p')), 
\end{equation}
the form of a transversal connection. The space of such connections is denoted by $G^2_{tr}$.  
\end{definition}
\subsection{Basis of transversal sections for a foliation} 
\label{basis}
In this subsection we recall \cite{CM} the notion of basis of transversal sections for a foliation. 
 Let $\mathcal M$ be a manifold of dimension $n$, equipped with a foliation $\F$ of codimension $q$.
\begin{definition}
A transversal section of $\F$ is an embedded $q$-dimensional submanifold $U\subset M$ which 
is everywhere transverse to the leaves. 
\end{definition}
\begin{definition}
If $\alpha$ is a path between two points $p_1$ and $p_2$ on the same leaf, 
and if $U_1$ and $U_2$ are transversal sections through $p_1$ and $p_2$, 
 then $\alpha$ defines a transport
along the leaves from a neighborhood of $p_1$ in $U_1$ to a neighborhood of $p_2$ in $U_2$, 
hence a germ of a diffeomorphism 
\[
hol(\alpha): (U_1, p_1)\hookrightarrow  (U_2, p_2),
\] 
called the holonomy of the path $\alpha$. 
\end{definition}
\begin{definition}
Two homotopic paths always define the same holonomy. 
If the above transport along $\alpha$ is defined in all of $U_1$ and embeds $U_1$ into $U_2$, this embedding
$h: U_1\hookrightarrow U_2$ will be denoted by 
$hol(\alpha): U_1\hookrightarrow U_2$. 
Embeddings of this form we called {holonomy embeddings}.  
\end{definition}
Note that composition of paths also induces an 
operation of composition on those holonomy embeddings. 
Transversal sections $U$ through $p$ as above should be thought of 
as neighborhoods of the leaf through $p$ in
the leaf space.
Then we have 
\begin{definition}
A {transversal basis} for $\mathcal M/\F$ as a 
family $\U$ of transversal sections $U\subset M$ 
with the property that, if $U_p$ is any transversal section through a given point $p\in M$, there exists a 
holonomy embedding $h: U\hookrightarrow V$ with $U\in \U$ and $p\in h(U)$.
\end{definition}
Typically, a transversal section is a $q$-disk given by a chart for the foliation. 
Accordingly, we can construct 
a transversal basis $\U$ out of a basis $\widetilde{\U}$ of $\mathcal M$ by domains of foliation charts
\[
 \phi_{U}: \widetilde{U}\tilde{\hookrightarrow } \mathbb{R}^{n-q}\times U,
\] 
$\widetilde{U}\in \widetilde{\U}$,
 with $U=\mathbb{R}^q$.
 Note that each inclusion $\widetilde{U}\hookrightarrow \widetilde{V}$
between opens of $\widetilde{\U}$ induces a holonomy embedding 
\[
h_{U, U_0}: U\hookrightarrow  U_0, 
\]
defined 
by the condition that the plaque in $\widetilde{U}$ through 
$p$ is contained in the plaque in $\widetilde{U_0}$ through $h_{U, U_0}(x)$. 

\subsection{${\bf \check C}$ech-de Rham cohomology in Crainic and Moerdijk construction}
Let us start with the first supporting example \cite{CM}. 
Recall the construction of the ${\rm \check C}$ech-de~Rham cohomology of a foliation.  
cohomology by Crainic and Moerdijk~\cite{CM}.
 Consider a foliation $\F$ of co-dimension $n$ on a smooth manifold $\mathcal M$.
Consider the double
complex 
\begin{equation}
\label{Cpq}
C^{k,l}=\prod_{U_0\stackrel{h_1}{\hookrightarrow }\cdots
\stackrel{h_k}{\hookrightarrow } U_k}\Omega^l(U_0),
\end{equation}
 where $\Omega^l(U_0)$ is the space of differential $l$-forms on $U_0$, 
and the 
product ranges over all $k$-tuples of holonomy embeddings between
transversal sections from a fixed transversal basis $\U$.
Component of $\varpi\in C^{k, l}$ are denoted by
$\varpi(h_1, \ldots, h_l)\in \Omega^l(\U_0)$.
 The
vertical differential is defined as 
\[
(-1)^k d:C^{k,l}\to
C^{k,l+1},
\]
 where $d$ is the usual de~Rham differential. 
The horizontal differential 
\[
\delta:C^{k,l}\to C^{k+1,l}, 
\]
 is given by
\[
\delta= \sum\limits_{i=1}^k (-1)^{i} \delta_{i},   
\]
\begin{equation}
\label{deltacpq}
\delta _{i} \varpi(h_1, \ldots , h_{k+1}) = G(h_1, \ldots , h_{k+1}),  
\end{equation}
where $G(h_1, \ldots , h_{k+1})$ is the multi-point connection of the form \eqref{locus}, i.e., 
\begin{equation}
\label{delta} 
\delta_{i} \varpi(h_1, \ldots , h_{p+1})= \left\{ \begin{array}{lll}
                                      h_{1}^{*}\varpi(h_2, \ldots , h_{p+1}),  \ \ \mbox{if $i=0,$}\\ 
                                      \varpi(h_1, \ldots, h_{i+1}h_{i}, \ldots, h_{p+1}),  \ \ \mbox{if $0<i< p+1,$}\\
                                      \varpi(h_1, \ldots, h_p), \ \ \mbox{if $i= p+1.$}
                        \end{array}
                \right.
\end{equation}
This double complex is actually a bigraded differential algebra, with the usual product 
\begin{equation}
\label{bigradif}
(\varpi \cdot  \eta)(h_1, \ldots , h_{k+k\,'})= (-1)^{kk\,'} \varpi(h_1, \ldots , 
h_{k}) \; h_{1}^{*} \ldots h_{k}^{*}\;.\eta(h_{k+1}, \ldots h_{k+k\,'}), 
\end{equation}
for $\varpi\in C^{k, l}$ and $\eta\in C^{k',l'}$, 
thus $(\varpi\cdot\eta)(h_1, \ldots , h_{k+k\,'}) \in C^{k+ k',l+ l'}$. 
The cohomology $\check{H}^*_\U(M/\F)$ of this complex is called the ${\rm \check C}$ech-de~Rham
cohomology of the leaf space $\mathcal M/\F$ with respect to the transversal basis $\U$. 
It is defined by 
\[
\check{H}^*_\U(M/\F)= {\mathcal Con}^{k+1}_{cl}(h_1, \ldots , h_{k+1})/ G^k(h_1, \ldots , h_{k}),  
\]
where ${\mathcal Con}^{k+1}_{cl}(h_1, \ldots , h_{k+1})$ is the space of closed multi-point connections, 
and $G^k$ $(h_1$, $\ldots$, $h_{k})$ is the space of $k$-point connection forms. 
In what follows we describe another supporting example of vertex algebra cohomology of codimension one foliations 
interpreted in terms of connections. 
\section{Spaces for double complexes}  
\label{spaces}
In this section we introduce the definition of spaces for a double complex suitable for the construction 
a grading-restricted vertex algebra cohomology for codimension one 
foliations on complex curves.  
A consideration of foliations of smooth manifold of arbitrary dimension will be given in elsewhere.   
Let $\U$ be a family of transversal sections of $\F$, (cf. \cite{CM} and Subsection \ref{basis}).  
We consider $(n,k)$-set of points, $n \ge 1$, $k \ge 1$,    
\begin{equation}
\label{points}
\left(p_1, \ldots, p_n; p'_1, \ldots, p'_k \right), 
\end{equation}
on a smooth manifold $\mathcal M$. 
Let us denote the set of corresponding local coordinates for $n+k$ points on $\mathcal M$ as $c_i(p_i)$, $1 \le i \le n+k 
$.    
In what follows we consider points \eqref{points} as points on either the leaf space $\mathcal M/\F$ of $\F$,
 or on transversal sections $U_j$ of the transversal basis $\U$. 
Since $\mathcal M/\F$ is not in general a manifold, one has to be careful in considerations of chains of local coordinates 
along its leaves \cite{IZ, Losik}.
 For association of formal parameters of mappings and vertex operators 
with points of $\mathcal M/\F$ we will use in what follows either their local coordinates on $\mathcal M$ or local  
coordinates on 
sections $U$ of a transversal basis $\U$ which are submanifolds of $\mathcal M$ of dimension equal to codimenion of a  
foliation $\F$.  
We will denote such local coordinates 
as $\{ l_i(p_i)\}$ (on $\mathcal M$), and $\{ t_i(p_i)\}$ (on $U$) correspondingly.   
In case of extremely singular foliations when it is not possible to use local coordinates of $\mathcal M$ in order to  
parameterize 
a point on $\mathcal M/\F$ we still able to use a local coordinate on a transversal section passing through this point on  
$\mathcal M/\F$.
In addition to that, note that the complexes considered below are constructed in such a way that one can always 
use coordinates on transversal sections only, avoiding any possible problems with localization of coordinates on leaves of $\mathcal M/\F$. 

For a $(n, k)$-set of a grading-restricted vertex algebra $V$ elements  
\begin{equation}
\label{vectors}
\left(v_1, \ldots, v_n; v'_1, \ldots, v'_k\right),   
\end{equation}
we consider linear maps 
\begin{equation} 
\label{maps}
\Phi: V^{\otimes n} \rightarrow \W_{z_1, \ldots, z_n}  
\end{equation}
 (see Appendix \ref{valued} for the definition of a 
$\W_{z_{1}, \dots, z_{n}}$ space), 
\begin{equation}
\label{elementw}
\Phi \left(dz_1^{\wt v_1} \otimes v_1, c_1(p_1); \ldots; dz_n^{\wt v_n} \otimes v_n,  c_n(p_1) \right),    
\end{equation}
where we identify, as it is usual in the theory of characters for vertex operator algebras on curves 
\cite{Y, Zhu, TUY, H2},   
 $n$ formal parameters   
$z_{1}, \ldots , z_{n}$ of $\W_{z_{1}, \ldots , z_{n}}$, 
with local coordinates $c_i(p_i)$ in vicinities of points $p_i$, $0 \le i \le n$,  
 on $\mathcal M$. 
 Elements $\Phi\in \W_{c_1(p_1), \ldots, c_n(p_n)}$ can be seen as 
  coordinate-independent $\overline{W}$-valued rational sections 
 of a vertex algebra bundle \cite{BZF} generalization.
The construction of vertex algebra cohomology of a foliation in terms of connections is parallel to ideas of \cite{BS}. 
 Such a construction will be explained elsewhere.  
Note that according to \cite{BZF} 
it can be treated as 
$\left({\rm Aut}\; \Oo^{(1)}\right)^{\times n}={\rm Aut}\; \Oo^{(1)} \times \ldots \times {\rm Aut}\; \Oo^{(1)}$-torsor 
of the product of groups of a coordinate transformation.   
In what follows, according to definitions of Appendix \ref{valued}, 
when we write an element $\Phi$ of the space $\W_{z_1, \ldots, z_n}$, we actually have in mind 
  corresponding matrix element $\langle w', \Phi\rangle$ that 
  absolutely converges (on certain domain) to 
a rational form-valued function $R(\langle w', \Phi\rangle)$. Quite frequently we will write 
$\langle w', \Phi\rangle$ which would denote a rational $\W$-valued form.   
In notations, we would keep tensor products of vertex algebra elements with $\wt$-powers of $z$-differentials 
when it is inevitable only. 

 Later in this section we prove, that
 for arbitrary $v_i$, $v'_j \in V$, $1 \le i \le n$, $1 \le j \le k$, 
 points $p'_i$ with local coordinates $c_i(p'_i)$ on transversal sections $U_i\in \U$ of $\F$, 
an element \eqref{elementw} 
 as well as the vertex operator 
 \begin{equation}
\label{poper}
\omega_W\left(dc_1({p'_i})^{\wt(v'_i)} \otimes v'_{i},  c_1({p'_1})\right)
= Y\left(  d(c_1(p'_i))^{\wt(v'_i)} \otimes v'_{i},  c_1({p'_i})\right), 
\end{equation}
 are invariant with respect to the action of 
$\left({\rm Aut}\; \Oo^{(1)}\right)^{\times n}$. 
In \eqref{poper} we mean usual vertex operator (as defined in Appendix \ref{grading}) not affecting the tensor product 
with corresponding differential. 
We assume that the maps \eqref{maps} are composable 
(according to Definition \eqref{composabilitydef} of Appendix \ref{composable}), 
 with $k$   
vertex operators $\omega_W(v'_i, c_i(p'_i))$, $1 \le i \le k$ 
with $k$ vertex algebra elements from \eqref{vectors}, and 
formal parameters associated with local coordinates on  
$k$ transversal sections of $\F$, of $k$ points from the set \eqref{points}. 

The composability of a map $\Phi$ with a number of vertex operators
consists of two conditions on $\Phi$.
 The first requires the existence of positive 
integers $N^n_m(v_i, v_j)$ depending just on $v_i$, $v_j$, and the second 
restricts orders of poles of corresponding sums \eqref{Inm}  
 and \eqref{Jnm}. 
Taking into account these conditions, we will see that 
the construction of the space \eqref{ourbicomplex} does depend on the choice 
of vertex algebra elements \eqref{vectors}. 

In this subsection we construct the spaces for the double complex defined for
codimension one foliations and associated to a grading-restricted vertex algebra. 
In order to keep control on the construction, we consider a section $U_j$ of a transversal basis $\U$,
and mappings $\Phi$ that belong to the space $\W_{c(p_1), \ldots, c(p_n)}$, 
depending on points $p_1, \ldots, p_n$ of intersection of $U_j$ with leaves of $\mathcal M/\F$ of $\F$.
It is assumed that local coordinates $c(p_1), \ldots, c(p_n)$ of points $p_i$, $1 \le i \le n$, 
are taken on $\mathcal M$ along these leaves of $\mathcal M/\F$. 
We then consider all together locally a collection of $k$ sections of $\U$.  
In order to define vertex algebra cohomology of $\mathcal M/\F$, 
mappings $\Phi$ are supposed to be composable with a number of vertex operators 
 with a number of vertex algebra elements, and 
formal parameteres identified with local coordinates of points $p'_1, \ldots, p'_k$ on each of 
$k$ transversal sections $U_j$, $1 \le j \le k$.    
The above setup is considered for a set of vertex algebra elements, which could be varied accordingly. 
We first introduce 
\begin{definition}
\label{initialspace}
 Let $p_1, \dots, p_n$ be points taken on the same transversal section $U_j \in \U$, $j\ge 1$. 
Assuming $k \ge 1$, $n \ge 0$,  
we denote by $C^n(V, \W, \F)(U_j)$, $0 \le j \le k$,       
the space of all linear maps \eqref{maps}
\begin{equation}
\label{mapy}
 \Phi: V^{\otimes n } \rightarrow \W_{c_1(p_1), \dots, c_{n}(p_n)},  
\end{equation}
composable with a $k$ of vertex operators \eqref{poper} with formal parameters identified 
with local coordinates $c_j(p'_j)$ functions around points $p'_j$ 
on each of transversal sections $U_j$, $1 \le j \le k$.  
\end{definition}
The set of vertex algebra elements  \eqref{vectors} 
plays the role of parameters in our further construction of 
the vertex algebra cohomology associated with a foliation $\F$. 
According to considerations of Subsection \ref{basis}, we assume that each transversal section of a transversal basis 
$\U$ possess a coordinate chart which is induced by a coordinate chart of $\mathcal M$ (cf. \cite{CM}).  

Recall the notion of a holonomy embedding (cf. Subsection \ref{basis}, cf. \cite{CM}) which maps 
 a section into another section of a transversal basis, 
 and a coordinate chart on the first section 
into a coordinate chart on the second transversal section.  
Motivated by  
the definition of the spaces for the ${\rm \check C}$ech-de Rham complex in \cite{CM} (see Subsection \ref{basis}),
 let us now introduce the following spaces: 
\begin{definition}
\label{defspace}
For $n\ge 0$, and $m\ge 1$, we define the space   
\begin{equation}
\label{ourbicomplex}
 {C}^{n}_{m}(V, \W, \U, \F) =  \bigcap_{ 
U_{1} \stackrel{h_1} {\hookrightarrow }  \ldots \stackrel {h_{m-1}}{\hookrightarrow } U_k  
\atop 1 \le j \le m }  
 C^{n}(V, \W, \F) (U_j),    
\end{equation}
where the intersection ranges over all possible $m$-tuples of holonomy embeddings $h_i$, $i\in \{1, \ldots, m-1\}$, 
between transversal sections of the baisis $\U$  for $\F$. 
\end{definition}
%
First, we have the following
\begin{lemma}
\label{empty}
\eqref{ourbicomplex} is non-zero.  
\end{lemma}
\begin{proof}
From the construction 
 of spaces for double complex for a grading-restricted vertex algebra cohomology, 
it is clear that the spaces $C^n (V, \W, \U, \F)(U_j)$,  $1 \le s \le m$ in Definition \ref{initialspace} are non-empty.
On each transversal section $U_s$,  $1 \le s \le m$, $\Phi(v_1, c_j(p_1);  \ldots; v_n, c_j(p_n))$
 belongs to the space 
$\W_{c_j(p_1),  \ldots, c_j(p_n)}$, and satisfy the $L(-1)$-derivative \eqref{lder1} and $L(0)$-conjugation
 \eqref{loconj}
properties. 
 A map $\Phi(v_1, c_j(p_1)$;  $\ldots$; $v_n, c_j(p_n))$ 
is composable with $m$ vertex operators 
with formal parameters identified with local coordinates $c_j(p'_j)$, on each transversal section $U_j$.
Note that on each transversal section, $n$ and $m$ the spaces \eqref{ourbicomplex} remain the same.   
The only difference may be constituted by the composibility conditions \eqref{Inm} and \eqref{Jnm} for  $\Phi$.

In particular, 
for  $l_{1}, \dots, l_{n}\in \Z_+$ such that $l_{1}+\cdots +l_{n}=n+m$, 
$v_{1}, \dots, v_{m+n}\in V$ and $w'\in W'$, recall \eqref{psii} that 
 \begin{eqnarray}
\label{psiii}
\Xi_{i}
&
=
&
\omega_V(v_{k_1}, c_{k_1}(p_{k_1})-\zeta_i)  \ldots  \omega_V(v_{k_i}, c_{k_i}(p_{k_i})-\zeta_i) \;\one_{V},     
\end{eqnarray}
where $k_i$ is defined in \eqref{ki}, 
for $i=1, \dots, n$, 
 depend on coordinates of points on transversal sections. 
At the same time, in the first composibility condition 
 \eqref{Inm} depends on projections $P_r(\Xi_i)$, $r\in \C$, 
 of $\W_{c(p_1), \ldots, c(p_n)}$ to $W$, and 
on arbitrary variables $\zeta_i$, $1 \le i \le m$.  
On each transversal connection $U_s$, $1 \le s \le m$,  
the absolute convergency is assumed for the series  \eqref{Inm} (cf. Appendix \ref{composable}). 
Positive integers $N^n_m(v_{i}, v_{j})$,
(depending only on $v_{i}$ and $v_{j}$) as well as $\zeta_i$, 
 for $i$, $j=1, \dots, k$, $i\ne j$, 
may vary for transversal sections $U_s$. 
Nevertheless, the domains of convergency determined by the conditions \eqref{granizy1} which have the form 
\begin{equation}
\label{popas}
|c_{m_i}(p_{m_i})-\zeta_{i}| 
+ |c_{n_i}(p_{n_i})-\zeta_{i}|< |\zeta_{i}-\zeta_{j}|,
\end{equation} 
for $m_i= l_{1}+\cdots +l_{i-1}+p$, $n=l_{1}+\cdots +l_{j-1}+q$,   
 $i$, $j=1, \dots, k$, $i\ne j$ and for $p=1, 
\dots,  l_i$ and $q=1, \dots, l_j$, 
are limited by $|\zeta_{i}-\zeta_{j}|$ in \eqref{popas} from above. 
Thus, for the intersection variation of sets of homology embeddings in \eqref{ourbicomplex}, 
 the absolute convergency condition for \eqref{Inm} is still fulfilled. 
Under intersection in \eqref{ourbicomplex}
by choosing appropriate $N^n_m(v_{i}, v_{j})$, 
one can analytically extend \eqref{Inm}  
to a rational function in $(c_1(p_{1}), \dots, c_{n+m}(p_{n+m}))$,
 independent of $(\zeta_{1}, \dots, \zeta_{n})$, 
with the only possible poles at 
$c_{i}(p_i)=c_{j}(p_j)$, of order less than or equal to 
$N^n_m(v_{i}, v_{j})$, for $i,j=1, \dots, k$,  $i\ne j$. 

As for the second condition in Definition of composibility, we note that, on each transversal section, 
the domains of absolute convergensy 
$c_{i}(p_i)\ne c_{j}(p_j)$, $i\ne j$
\[
|c_{i}(p_i)| > |c_{k}(p_j)|>0,
\]
 for 
 $i=1, \dots, m$, 
and 
$k=1+m, \dots, n+m$, 
for 
\begin{eqnarray}
\mathcal J^n_m(\Phi) &=&  
\sum_{q\in \C}\langle w',
\omega_W(v_{1}, c_1(p_1)) \ldots  \omega_W(v_{m}, c_m(p_m)) \;    
\nn
&& \qquad 
P_{q}(\Phi( v_{1+m}, c_{1+m}(p_{1+m}); \ldots; v_{n+m}, c_{n+m}(p_{n+m})) \rangle, 
\end{eqnarray}
are limited from below by the same set ot absolute values of local coordinates on 
transversal section. 
Thus, under intersection in \eqref{ourbicomplex} this condition is preserved, and 
  the sum \eqref{Jnm} can be analytically extended to a
rational function 
in $(c_{1}(p_1)$, $\dots$, $c_{m+n}(p_{m+n}))$ with the only possible poles at
$c_{i}(p_i)=c_{j}(p_j)$, 
of orders less than or equal to 
$N^n_m(v_{i}, v_{j})$, for $i, j=1, \dots, k$, $i\ne j$. 
Thus, we proved the lemma.   
\end{proof}

\begin{lemma}
\label{nezu}
The double complex \eqref{ourbicomplex} does not depend on the choice of transversal basis $\U$. 
\end{lemma}
The main statement of this section is contained in the following
\begin{proposition} 
\label{nezc}
For a quasi-conformal grading-restricted vertex algebra $V$ and it module $W$,    
the construction \eqref{ourbicomplex} is canonical, i.e.,
 does not depend on foliation preserving choice of local coordinates on $\mathcal M/\F$.   
\end{proposition}
The proofs of Lemmas \ref{empty}, \ref{nezu}, and Proposition \ref{nezc} is contained in Appendix \ref{proof}. 
\begin{remark}
The condition of quasi-conformality is necessary in the proof of invariance of elements of the space 
$\W_{z_1, \ldots, z_n}$ 
with respect to a vertex algebraic representation (cf. Appendix \ref{grading}) of the group 
$\left({\rm Aut}\; \Oo^{(1)}\right)^{\times n}$. 
In what follows, when it concerns the spaces \eqref{ourbicomplex} we will always assume the quasi-conformality of $V$. 
\end{remark}
 A generalization of proofs of Lemmas \ref{empty}, \ref{nezu}, \ref{subset} and Proposition \ref{nezc}
 for the case of a arbitrary codimension foliation on a smooth $n$-dimensional manifold 
will be given elsewhere. 
The proof of Proposition \ref{nezc} is contained in Appendix \ref{proof}. 

Let $W$ be a grading-restricted $V$ module. 
Since for $n=0$, $\Phi$ does not include variables, and  
 due to Definition \ref{composabilitydef} of the composability, we can put: 
\begin{equation}
\label{proval}
{C}_{k}^{0}(V, \W, \F)= W,  
\end{equation}
for $k\ge 0$. Nevertheless, according to our Definition \ref{ourbicomplex}, mappings that belong to \eqref{proval}
are assumed to be composable with a number of vertex operators with depending on local coordinates 
of $k$ points on $k$ transversal sections.  
\begin{lemma}
\label{subset}
\begin{equation}
\label{susus}
{C}_{m}^{n}(V, \W, \F)\subset {C}_{m-1}^{n}(V, \W, \F).     
\end{equation}
\end{lemma}
The proof of this Lemma is contained in Appendix \ref{proof}. 
For our further purposes we have to define spaces suitable for the definition of 
a double complex with a fixed point. 
Such double complex will be needed for the construction of 
vertex algebra cohomologies of $\mathcal M/\F$, in particular, for $H^1_m(V, \W, \F)$, $m \ge 0$ (see Section \ref 
{coboundary}).  
\begin{definition}
\label{fixedpointphi}
Let us fix a point $p'_r$ on the transversal section $U_r \in \U$, $r \ge 1$. 
Assuming $k \ge 0$, $n \ge 0$, consider the space 
$C^n(p'_r; V, \W, \F)(U_r)$, of linear mappings    
\begin{equation}
\label{mapy}
 \Phi: V^{\otimes n} \rightarrow \W_{c_1(p_1), \ldots,  c_{n}(p_n)},   
\end{equation}
composable with $k$ vertex operators with formal parameters identified with local coordinates
   $\{c_1(p'_1), \ldots, c_r(p'_r)|_{p'_r}, \ldots, c_n(p'_k) \}$, 
 on each of 
$k$ sections on a transversal basis $\U$. 
\end{definition}
We assume that, for the intersection points $p'_j$ of each $U_j$, $1 \le j \le k$,  
 we are able to define a coordinate function $t_j(p'_j)$ around $p'_j$ such that it covers the whole $U_j$. 
Thus, the holonomy embeddings $h_j$ deliver a map of local coordinate functions and vertex operators \ref{poper},   
\[
 h_j: t_j(p'_j) \to  t_{j+1}(p'_{j+1}), 
\]
and we have a sequence of mappings 
\begin{equation}
\label{holseq}
p'_{1} \stackrel{h_1} {\to }  \ldots \stackrel{h_{r-1}} {\to} p'_r 
\stackrel{h_{r}} {\to}   \ldots  \stackrel {h_{k-1}}{\to } p_k.  
\end{equation}
Let us now introduce the following spaces: 
\begin{definition}
\label{defcomplex}
For $n\ge 0$, and $k\ge 1$, consider the space  
\begin{equation}
\label{ourbicomplex11}
 {C}^{n}_{k}(p'_r; V, \W, \U, \F) =  \bigcap_{ 
p'_{1} \stackrel{h_1} {\to }  \ldots \stackrel{h_{r-1}} {\to} p'_r 
\stackrel{h_{r}} {\to}   \ldots  \stackrel {h_{k-1}}{\to } p_k
\atop j\in \{1, \ldots, k\} } 
 C^{n}(p'_r; V, \W, \F) (U_j),     
\end{equation}
where the intersection ranges over all possible $k$-sequences \eqref{holseq} of holonomy mappings 
 $h_i$, $i\in \{1, \ldots, k\}$  
among points on transversal sections of the basis $\U$ with the fixed point $p_r$.   
\end{definition}
\begin{remark}
\label{remseq}
Note that in this Definition it is assumed that points $\{p'_j\}$, $1 \le {r-1}$, $ r+1 \le k$, 
  in sequences \eqref{holseq} of holonomy mappings  
are free for moving along corresponding sections $\{ U_j\}$ 
  of $\U$. 
\end{remark}
Similar to the proof of Proposition \ref{nezc}, one shows that the spaces \eqref{ourbicomplex11} do not depend on the 
choice of $\U$ and coordinates on $\mathcal M/\F$, thus we then omit $\U$ in notations. 
\section{Coboundary operators and cohomology of codimension one foliations}
\label{coboundary}
Though the question of introduction a product on the space $\W_{z_1, \ldots, z_n}$ is not developed yet 
(private communication with Y.-Zh. Huang). 
Nevertheless, one can introduce an internal product $\circ$ which allow to look at the action of 
 coboundary operator as the integrability condition on one forms defining a foliation
 and related to the Godbillon-Vey cohomology class for codimension one foliations \cite{G}.     

The coboundary operator is introduced as the form of a multi-point vertex algebra connection
(cf. Definition \ref{locus} in Section \ref{connections}): 
\begin{equation}
\label{deltaproduct}
\delta^n_m \Phi = G(p_1, \ldots, p_{n+1}),     
\end{equation}
\begin{eqnarray}
\label{hatdelta1}
 G(p_1, \ldots, p_{n+1}) 
&=& 
\langle w', \sum_{i=1}^{n}(-1)^{i} \; \Phi\left( \omega_V\left(v_i, c_i(p_i) - c_{i+1}(p_{i+1})) 
 v_{i+1} \right)  
 \right) \rangle,  
\\
\nonumber
&+& \langle w', \omega_W \left(v_1, c_1(p_1)  \right) \; \Phi (v_2, c_2(p_2); \ldots; v_n, c_n(p_n) )  
\rangle   
\nn
\nonumber
 &+& (-1)^{n+1} 
\langle w', 
 \omega_W(v_{n+1}, c_{n+1} (p_{n+1})) 
\; \Phi(v_1, c_1(p_2); \ldots; v_n, c_n(p_n)) \rangle,  
\end{eqnarray}
for arbitrary $w' \in W'$ (dual to $W$).
Note that it is assumed that the coboundary operator does not affect $dc(p)^{\wt(v_i)}$-tensor multipliers in $\Phi$. 
\begin{remark}
Following the construction of \cite{Huang},
 let us introduce the coboundary operator 
${\delta}^{n}_{k}$ acting on elements $\Phi \in C^{n}_{k}(V, \W, \F)$ of the spaces \eqref{ourbicomplex}, 
which has the integrability conditition form 
 \cite{G} with respect to a product: 
\begin{equation}
\label{deltaproduct}
\delta^n_m \Phi =\mathcal E \circ \Phi,   
\end{equation}
where 
\begin{eqnarray*}
\mathcal E &=&  \left( E^{(1)}_W,\; \sum\limits_{i=1}^n (-1)^i \; E^{(2)}_{V; \one_V},  \; E^{W; (1)}_{W V} \right),  
\end{eqnarray*} 
 the product $\circ$ is given by 
\[
\circ = \sum\limits_{j=0}^{n+1} \circ_{j},     
\]
where the elements
 $E^{(1)}_{W}$, $E^{W; (1)}_{WV}$,  $E^{(2)}_{V; \one_V}$, and multiplications  
$\circ_i$ are defined in Appendix \ref{properties}, 
\begin{eqnarray}
\label{hatdelta}
\nonumber 
{\delta}^{n}_{k} \Phi& =& E^{(1)}_{W}\circ_{0
} \Phi  
+\sum_{i=1}^{n}(-1)^{i} E^{(2)}_{V; \one_V}   \circ_{i} \Phi
 + (-1)^{n+1} 
E^{W; (1)}_{WV}\circ_{m+1
} \Phi. 
\end{eqnarray}
\end{remark}
%
\begin{remark}
Inspecting construction of the double complex spaces \eqref{ourbicomplex}
 we see that the action \eqref{hatdelta1} of the 
$\delta^n_m$ on an element of $C^n_m(V, \W, \F)$ provides a coupling (in terms of $\W_{z_1, \ldots, z_n}$-valued 
rational functions) of vertex operators
 taken at the local coordinates 
$c_i(z_{p_i})$, $0 \le i \le k$,  
at the vicinities of the same points $p_i$ taken on transversal sections for $\F$,  
with elements of $C^n_{m-1}(V, \W, \F)$ taken at   
points at the local coordinates 
$c_i(z_{p_i})$, $0 \le i \le n$ on $\mathcal M$ for points $p_i$ considered on the leaves of $\mathcal M/\F$.  
\end{remark}
\begin{remark}
We also mention that \eqref{deltaproduct} can be written completely in terms of intertwining operators (cf. Appendix \ref 
{grading}) 
in the form 
\[
\delta^n_m \Phi= \sum\limits_{i=1}^3  
\langle w', 
  e^{\xi_i L_W(-1)}\; \omega^{W}_{WV} \left( \Phi_i \right) u_i \rangle, 
\]
for some $\xi_i \in \C$, and $u_i \in V$, and $\Phi_i$ obvious from \eqref{deltaproduct}.  
Namely, 
\begin{eqnarray*}
&& \delta^n_m \Phi=  
 \langle w',  e^{ c_1(p_1) L(-1)_{W} } \;  \omega^{W}_{WV}   
\left(  \Phi \left( v_2, c_2(p_2); \ldots; v_n, c_{n+1} (p_{n+1}), - c_1(p_1)  \right) v_1  \right.   
\rangle 
\nn
&&
+\sum\limits_{i=1}^n (-1)^i e^{\zeta L_W(-1)}   
\langle w', \omega^{W}_{WV} \left(  
\Phi \left( \omega_V \left(v_i, c_i(p_i) - c_{i+1}(p_{i+1})\right), -\zeta \right)  \one_V \rangle \right) 
\nn
&&
=   
 \langle w',  e^{ c_{n+1}(p_{n+1}) L(-1)_{W} } \;  \omega^{W}_{WV}    
\left( \Phi \left( v_1, c_1(p_1); \ldots; v_{n}, c_n (p_n), - c_{n+1}(p_{n+1})  \right) v_{n+1}     
\rangle \right),   
\end{eqnarray*}
for an arbitrary $\zeta\in \C$. 
\end{remark}
\subsection{Complexes of transversal connection}  
\label{comtrsec}
In addition to the double complex $(C^n_m(V$, $\W$, $\F)$, $\delta^n_m)$ provided by \eqref{ourbicomplex} and 
\eqref{deltaproduct},
 there exists an exceptional short double complex which we call transversal connection complex.  
 We have  
\begin{lemma}
\label{lemmo}
 For $n=2$, and $k=0$, there exists a subspace
 $C^{0}_{ex}(V, \W, \F)$ 
\[
{C}_{m}^{2}(V, \W, \F) \subset C^{0}_{ex}(V, \W, \F) \subset C_{0}^{2}(V, \W, \F), 
\]
 for all $m \ge 1$, with the action of coboundary operator 
${\delta}^{2}_{m}$ defined.   
\end{lemma}
\begin{proof}
Let us consider the space $C_{0}^{2}(V, \W, \F)$.
%
vertex operators composable. 
Indeed, the space $C_{0}^{2}(V, \W, \F)$ contains elements of $\W_{c_1(p_1), c_2(p_2)}$ so that the action  
of $\delta_{0}^{2}$ is zero. 
Nevertheless, as for $\mathcal J^n_m(\Phi)$ in \eqref{Jnm}, Definition \ref{composabilitydef}, let us consider sum of  
projections 
\[
P_r: \W_{z_i, z_j} \to W_r, 
\]
for $r\in \C$, and $(i, j)=(1,2), (2, 3)$,   
so that the condition \eqref{Jnm}  is satisfied for some connections similar to the action \eqref{Jnm} of $\delta^2_0$. 
Separating the first two and the second two summands in \eqref{hatdelta1}, we find that for a subspace 
of ${C}_{0}^{2}(V, \W, \F)$, which we denote as 
${C}_{ex}^{2}(V, \W, \F)$, consisting of three-point connections $\Phi$ such that
 for $v_{1}$, $v_{2}$, $v_{3} \in V$, 
$w'\in W'$, and arbitrary $\zeta \in \C$, the following forms of connections 
\begin{eqnarray}
\label{pervayaforma}
&& G_1(c_1(p_1), c_2(p_2), c_3(p_3))
\nn
&&
= \sum_{r\in \C} \Big( \langle w', E^{(1)}_{W} \left(  v_{1}, c_1(p_{1}\right);      
P_{r} \left(   \Phi\left(v_{2}, c_2(p_{2})-\zeta;  v_{3}, c_3(p_{3}) - \zeta  \right)  \right) \rangle   
\nn
&&\quad +\langle w', 
\Phi \left( v_{1}, c_1(p_{1});  P_{r} \left(E^{(2)}_{V}
\left(v_{2}, c_2(p_{2})-\zeta; v_{3}, c_3(z_{3} \right)-\zeta; \one_V \right),    
 \zeta \right) 
\rangle \Big) 
\nn
&&
\nn
&& = \sum_{r\in \C} \big( \langle w', \omega_{W} \left( v_{1}, c_1(p_{1}) \right)\;       
 P_{r}\left( \Phi\left(v_{2}, c_2(p_{2})-\zeta; v_{3}, c_3(p_{3}) - \zeta \right) \right) \rangle 
\nn
&&\quad +\langle w', 
\Phi \left( v_{1}, c_1(p_{1});  P_{r} \left(  \omega_V  
\left(v_{2}, c_2(p_{2})-\zeta\right)  \omega_V \left( v_{3}, c_3(z_{3})-\zeta \right) \one_V \right),     
 \zeta \right)
\rangle\big),  
\nn
\end{eqnarray}
and 
\begin{eqnarray}
\label{vtorayaforma}
&&
G_2( c_1(p_1), c_2(p_2), c_3(p_3) ) 
\nn 
 && \qquad =\sum_{r\in \C} \Big( \langle w', 
 \Phi \left(  P_{r} \left(   E^{(2)}_{V} \left( v_{1}, c_1(p_{1}) -\zeta, v_2, c_2(p_{2})-\zeta; \one_V \right) \right), 
\zeta;  
 v_{3}, c_3(p_{3}) \right) \rangle  
\nn
&&\quad + \langle w',  
E^{W; (1)}_{WV} \left(    P_{r} \left(    \Phi \left(   v_{1},  c_1(p_{1})-\zeta;   v_{2}, c_2(p_{2}) -\zeta \right), 
\zeta ;   
 v_{3}, c_3(p_{3}) \right) \right) \rangle \Big) 
\nn
&&
\nn
&&
 = \sum_{r\in \C}\big(\langle w', 
 \Phi \left(P_{r}( \omega_V \left(v_{1}, c_1(p_{1})-\zeta\right) \omega_V\left(v_2, c_2(p_{2})-\zeta) \one_V, 
\zeta \right));  
 v_{3}, c_3(p_{3}) \right) \rangle  
\nn
&&\quad + \langle w', 
 \omega_V \left(v_{3}, c_3(p_{3}) \right) \; P_{r} \left(    \Phi \left(v_{1}, c_1(p_{1})-\zeta; v_{2}, c_2(p_{2})-\zeta
 \right)  
\right)
  \rangle \big),  
\end{eqnarray}
are absolutely convergent in the regions 
\[
|c_1(p_{1})-\zeta|>|c_2(p_{2})-\zeta|,  
\]
\[
 |c_2 (p_{2})-\zeta|>0, 
\]
\[
|\zeta-c_3(p_{3})|>|c_1(p_{1})-\zeta|, 
\]
\[
|c_2(p_{2})-\zeta|>0,
\]
where $c_i$, $1 \le i \le 3$ are coordinate functions, 
 respectively, and can be analytically extended to  
rational form-valued functions in $c_1(p_{1})$ and $c_2(p_{2})$ with the only possible poles at
$c_1(p_{1})$, $c_2(p_{2})=0$, and $c_1(p_{1})=c_2(p_{2})$.
Note that \eqref{pervayaforma} and \eqref{vtorayaforma} constitute the first two and the last two terms of 
\eqref{hatdelta1} correspondingly.  
According to Proposition \ref{comp-assoc} (cf. Appendix \ref{composable}),  
 ${C}_{m}^{2}(V, \W, \F)$ is a subspace of ${C}_{ex}^{2}(V, \W, \F)$,  for $m\ge 0$,
and $\Phi \in {C}_{m}^{2}(V, \W, \F)$ are composable with $m$ vertex operators. 
Note that \eqref{pervayaforma} and \eqref{vtorayaforma} represent sums of forms $G_{tr}(p, p')$ 
of transversal connections 
\eqref{transa} (cf. Section \ref{firstcohomology}). 
\end{proof}
\begin{remark}
It is important to mention that, according to general principle, abserved in \cite{BG}, 
for non-vanishing $G(c(p), c(p'), c(p''))$,  
there exists an invariant structure, e.g., a cohomological class. 
In our case, it appears as a non-zero subspaces 
${C}_{m}^{2}(V, \W, \F) \subset {C}_{ex}^{2}(V, \W, \F)$ in ${C}_{0}^{2}(V, \W, \F)$. 
\end{remark}

Then we have 
\begin{definition}
\label{cobop}
 The coboundary operator 
\begin{equation}
\label{halfdelta}
{\delta}^{2}_{ex}: {C}_{ex}^{2}(V, \W, \F)
\to {C}_{0}^{3}(V, \W, \F),
\end{equation}
is defined 
 by 
three point connection of the form 
\begin{equation}
\label{ghalfdelta}
\delta^2_{ex} \Phi= G_{ex}(p_1, p_2, p_3),  
\end{equation}
\begin{eqnarray}
\label{ghalfdelta1}
G_{ex}(p_1, p_2, p_3) &=&  
\langle w', \omega_{W} \left(v_{1}, c_1(p_1)\right) \; 
\Phi\left(v_{2}, c_2(p_2); v_{3}, c_3(p_3) \right) \rangle 
\nn
&&
\quad \quad 
- 
 \langle w',  
\Phi \left(  \omega_{V}( v_{1}, c_1(p_1))\; \; \omega_V (v_{2}, c_2(p_2) ) \one_V; v_{3}, c_3(p_3) \right)  
\rangle
\nn
 &&\quad
+ 
 \langle w', \Phi(v_{1}, c_1(p_1);  \; \omega_{V} (v_{2}, c_2(p_2))\;  \omega_V( v_{3}, c_3(p_3)) \one_V) 
 \rangle 
\nn
 &&\quad \quad 
+ 
 \langle w',  
 \omega_W (v_{3}, c_3(p_3)) \; \Phi \left(v_{1}, c_1(p_1); v_{2}, c_2(p_2) \right)  
\rangle,
\end{eqnarray}
for $w'\in W'$,
$\Phi\in {C}_{ex}^{2}(V, \W, \F)$,
$v_{1}, v_{2}, v_{3}\in V$ and $(z_{1}, z_{2}, z_{3})\in F_{3}\C$.   
\end{definition}
\begin{remark}
Similar to \eqref{deltaproduct}, \eqref{ghalfdelta} can also be written in Frobenius form: 
\begin{eqnarray}
\label{halfdelta1}
&& {\delta}^{2}_{ex} \Phi = \mathcal E_{ex} \circ_{ex} \Phi, 
\end{eqnarray}
where 
\begin{equation}
\label{mathe2}
\mathcal E_{ex} = \left( E^{(1)}_W,\; \sum\limits_{i=1}^2 (-1)^n E^{(2)}_{V; \one_V}, \; E^{W; (1)}_{W V} \right),  
\end{equation}
with the product 
\[
\circ_{ex}= \sum_{i=0}^3 \circ_j, 
\]
\begin{eqnarray*}
{\delta}^{2}_{ex} \Phi &=& E^{(1)}_{W} \circ_{0} \Phi 
+ \sum\limits_{i=1}^2 (-1)^i  E^{(2)}_{V, \one_V}  \circ_i \Phi
+ E^{W; (1)}_{WV}\circ_{3}\Phi,
\nn
 &
=&
\langle w', E^{(1)}_{W} \left(  v_{1}, c_1(p_1); \; 
\Phi  \left(v_{2}, c_2(p_2); v_{3}, c_3(p_3) \right) \rangle \right)
\nn
&&
\quad \quad 
- 
 \langle w',  
\Phi \left(E^{(2)}_{V}(v_{1}, c_1(p_1);   v_{2}, c_2(p_2); \one_V ; v_{3}, c_3(p_3)\right) 
\rangle
\nn
 &&\quad 
+ 
\langle w', \Phi(v_{1}, c_1(p_1); E^{(2)}_{V}(v_{2}, c_2(p_2); v_{3}, c_3(p_3); \one_V) 
 \rangle 
\nn
 &&\quad \quad 
+ 
 \langle w',  
E^{W; (1)}_{WV} \; \Phi \left(v_{1}, c_1(p_1); v_{2}, c_2(p_2)); v_{3}, c_3(p_3) \right) 
\rangle. 
\nn 
&& 
\nn
&&
\end{eqnarray*}
\end{remark}
Then we have 
\begin{proposition}
\label{cochainprop}
The operators \eqref{hatdelta} and \eqref{halfdelta} provide the chain-cochain complexes    
\begin{equation}
\label{conde}
{\delta}^{n}_{m}: C_{m}^{n}(V, \W, \F)  
\to C_{m-1}^{n+1}(V, \W, \F),   
\end{equation}  
\begin{equation}
\label{deltacondition}  
{\delta}^{n+1}_{m-1} \circ {\delta}^{n}_{m}=0,  
\end{equation} 
\[
{\delta}^{2}_{ex}\circ {\delta}^{1}_{2}=0, 
\]
\begin{equation}\label{hat-complex}
0\longrightarrow {C}_{m}^{0}(V, \W, \F) 
\stackrel{{\delta}^{0}_{m}}{\longrightarrow}
 {C}_{m-1}^{1}(V, \W, \F)  
\stackrel{{\delta}^{1}_{m-1}}{\longrightarrow}\cdots 
\stackrel{{\delta}^{m-1}_{1}}{\longrightarrow}
 {C}_{0}^{m}(V, \W, \F)\longrightarrow 0, 
\end{equation}
\begin{equation}\label{hat-complex-half}
0\longrightarrow {C}^{0}_{3}(V, \W, \F)
\stackrel{{\delta}_{3}^{0}}{\longrightarrow} 
 {C}^{1}_{2}(V, \W, \F)
\stackrel{{\delta}_{2}^{1}}{\longrightarrow}{C}_{ex}^{2}(V, \W, \F)
\stackrel{{\delta}_{ex}^{2}}{\longrightarrow} 
 {C}_{0}^{3}(V, \W, \F)\longrightarrow 0,
\end{equation}
on the spaces \eqref{ourbicomplex}. 
\end{proposition}
Since  
\[
{\delta}_{2}^{1} \;  {C}_{2}^{1}(V, \W, \F) \subset 
 {C}_{1}^{2}(V, \W, \F)\subset 
 {C}_{ex}^{2}(V, \W, \F),
\]
the second formula follows from the first one, and 
\[
{\delta}^{2}_{ex}\circ  {\delta}^{1}_{2}
= {\delta}^{2}_{1}\circ  {\delta}^{1}_{2} 
=0.
\]
\begin{proof}
The proof of this proposition is analogous to that of Proposition (4.1) in \cite{Huang} 
for chain-cochain complex of a grading-restricted vertex algebra.  
The only difference is that we work with the space $\W_{c_1(p_1), \ldots, c_n(p_n)}$ 
instead of
 $W_{z_1, \ldots, z_n}$. 
\end{proof}
\subsection{Fixed point double complexes}
Recall the definition of fixed point double complex spaces \eqref{ourbicomplex11} in Section \ref{spaces}. 
Then we have the following 
\begin{lemma}
\label{complexlemma}
The double complex $\left(C^n_k(p'_r; V, \W, \F), \delta^n_k|_{p'_r} \right)$
is a subcomplex of double chain-cochain complex $\left(C^n_k(V, \W, \F), \delta^n_k \right)$.  
\end{lemma}
\begin{proof}
According to Remark \ref{remseq} in Section \ref{spaces}, we assume that in the construction of \eqref{ourbicomplex}, 
the points $\{p'_j\}$, $1 \le {r-1}$, $ r+1 \le k$  
 in sequences \eqref{holseq} of holonomy mappings  
 move freely  along corresponding sections 
$\{ U_j\} \in \U$.  
In the intersection in Definition \eqref{ourbicomplex} of $C^n_k(V, \W, \F)$, 
the points $\{p'_j\}$, $1 \le {r-1}$, $ r+1 \le k$ exhaust corresponding sections $\{ U_j\}$ of $\U$. 
Thus, 
\[
C^n_k(p'_r; V, \W, \F) \subset C^n_k(V, \W, \F). 
\]
It is clear that the operator $\delta^n_k|_{p'_r}$ is a reduction of 
$\delta^n_k$, and satisfies the chain-cochain property as in Proposition \eqref{cochainprop}.  
Thus the Lemma is proved. 
\end{proof}
\subsection{Cohomology}
Now let us define 
 the cohomology  
of the leaf space $\mathcal M/\F$ associated with a grading-restricted vertex algebra $V$. 
\begin{definition}
\label{defcohomology}
 We define the 
 $n$-th cohomology $H^{n}_{k}(V, \W, \F)$ of $\mathcal M/\F$ with coefficients in   
$\W_{z_1, \ldots, z_n}$ (containing maps composable $k$ vertex operators on $k$ transversal sections)    
to be the factor space of closed 
multi-point connections by the space of connection forms: 
\begin{equation}
\label{cohom1}
 H_{k}^{n}(V, \W, \F)= {\mathcal Con}_{k; \; cl}^n/G^{n-1}_{k+1}. 
\end{equation}
\end{definition}
Note that due to \eqref{hatdelta1}, \eqref{ghalfdelta1}, 
and Definitions \ref{locus} and \ref{gform} (cf. Section \ref{coboundary}), 
  it is easy to see that \eqref{cohom1} is equivalent to the standard cohomology definition  
\begin{equation}
\label{cohom}
 H_{k}^{n}(V, \W, \F)= \ker  \delta^{n}_{k}/\mbox{\rm im}\; \delta^{n-1}_{k+1}. 
\end{equation}  
Recall Definition \ref{ourbicomplex11} of fixed-point double complex spaces $C^n_k(p'_r; V, \W,  \F )$.
 Simultaneously with Definition \ref{defcohomology}, we formulate
\begin{definition}
\label{fixedpointcohomology}
Let $U_r$, $r \ge 1$, be a section of a basis $\U$, and $p'_r \in U_r$ be a fixed point.  
Here we define the fixed point cohomology as 
\begin{equation}
\label{}
 H_{k}^{n}(V, \W, \F)= {\mathcal Con}_{p;\; k; \; cl}^n/ G_{p; \; k+1}^{n-1}, 
\end{equation}
which is equivalent to 
\[
H^n_k \left(p'_r; V, \W, \F \right)= {\rm Ker}\; \delta^n_k / {\rm Im} \;\delta^n_k|_{p'_r}.    
\]  
\end{definition}
From Lemma \ref{complexlemma} it follows 
\begin{lemma}
The cohomology $H^n_m \left(p; V, \W,  \F \right)$ is given by 
\[
H^n_m \left(V, \W,  \F \right) = \bigcup\limits_{p'_r \in U_r} H^n_m \left(p'_r; V, \W,  \F \right). 
\]
\end{lemma}
\subsection{Relations to Crainic and Moerdijk construction} 
\label{relcm}
In particular, we have the following 
\begin{lemma}
In codimension one case, 
the construction of the double complex $(C^{k,l}), \delta)$,  \eqref{Cpq}, \eqref{deltacpq} 
 follows from the construction of the double complex  
$(C_m^n(V, \W, \F), \delta^n_m)$ of \eqref{hat-complex}. 
Thus, the {${\rm \check C}$}ern-de Rham cohomology of a foliated smooth manifold 
results from grading-restricted vertex algebra $V$ cohomology.    
\end{lemma}
\begin{proof}
One constructs the space of differential forms of degree $k$ by 
elements $\Phi$ of $C^n_m(V, \W, \F)$ 
\begin{eqnarray}
\label{bomba1}
 \langle w', \Phi \left( dc_1(p_1)^{\wt(v_1)} \otimes v_1, c_1(p_1); \ldots ; dc_n(p_n)^{\wt(v_n)} v_{n}, 
c_n(p_n)\right) \rangle, 
\end{eqnarray} 
such that $n=k$ the total degree 
\[
\sum\limits_{i=1}^n \wt(v_i) =  l,    
\]
$v_i \in V$.     
The condition of composability of $\Phi$ with $m$ vertex operators allows us make the association of 
the differential form $\varpi(h_1,\ldots, h_{n})$ with \eqref{bomba1}
$(h^*_1, \ldots, h^*_k)$ with $(v_i, \ldots, v_k)$, and to   
represent 
a sequence of holomorphic embeddings $h_1, \ldots, h_p$ for $U_0, \ldots, U_p$ in \eqref{Cpq} by 
vertex operators $\omega_W$, i.e, 
\[
\left( h(h^*_1) \ldots h(h^*_{n}) \right)(z_1, \ldots, z_n)) 
=  \omega_W\left(v_1, t_1(p_1)  \right) 
\ldots \omega_W\left(v_l, t(p_n) \right).   
\]
Then, by using Definitions of coboundary operator \eqref{deltaproduct},   
  we see that the definition of the coboundary operator of \cite{CM} is parallel to 
the definition \eqref{hatdelta}.
\end{proof}
\section{First cohomologies $H^1_m(V, \W, \F)$ of codimension one foliations}  
\label{firstcohomology}
In \cite{Hu3}, lower cohomologies for a grading-restricted vertex algebra were computed. 
In this paper we find the first grading-restricted vertex algebra cohomologies $H^1_m(V, \W, \F)$ 
 and the second cohomology $H^2_{ex}(V, \W, \F)$ 
 for a codimension one foliation 
$\F$.  
Let us first consider one-variable reduction of multi-point connections. 
Such reduction is called in \cite{Hu3} a derivation. 
In analogy with a definition of \cite{Huang}, we introduce the following definition of the derivation 
applicable to maps from $V$ to $\W$. 
\begin{definition}
Let $V$ be a grading-restricted vertex algebra and $W$ a $V$-module. 
A grading-preserving 
linear map 
\[
g: V \to \W, 
\]
 is called a derivation if   
\begin{eqnarray*}
g\left(\omega_{V}(u, z)v, 0 \right) &= & e^{zL_{W}(-1)}\omega_{W}(v, -z)\;g(u, 0)+\omega_{W}(u, z)\;g(v, 0)  
\nn
&=& 
\omega_{WV}^{W}(g(u, 0), z)v  + \omega_{W}(u, z)\;g(v, 0),   
\end{eqnarray*}
for $u$, $v\in V$, where 
$\omega_{WV}^{W}(v, z)$ is the intertwiner-valued vertex operator    
 in accordance with notaions of \eqref{poper}.  
We use $\mbox{\rm Der}\;(V, \W)$ to denote the space of all such derivations. 
It is clear that 
\[
g(v, 0)=\mathcal G(v, 0).  
\]
\end{definition}
As we see from the definition of a derivation over $V$, it depends on one element of $V$ only. 
The space of one $V$-element two
point holomorphic connections reduces to the space  
of derivations over $\W$ \cite{Huang}.  
In \cite{Hu3} it is proven the following 
\begin{lemma}
\label{f-one}
Let $g(v, 0): V\to \W$ be a derivation. Then $g(\one_V,0)=0$. 
\end{lemma}
We will need another statement proven in \cite{Hu3}
\begin{lemma}
Let
\[
\Phi: V \to \W_{z}, 
\] 
be an element of $C_{m}^{1}(V, \W, \F)$  
satisfying 
\[
\delta_{m}^{1} \Phi=0.
\]
Then 
$\Phi(v, 0)$ is a grading-preserving
linear map from 
$V$ to $\W$, i.e.,  
\begin{eqnarray*}
z^{L(0)}\Phi(v, 0) = \Phi(z^{L(0)}v, 0) 
=
z^{n}\Phi(v, 0).  
\end{eqnarray*}
\end{lemma}
In \cite{Hu3}, the first cohomologies $H^1_M(V, W)$ of a grading-restricted vertex algebra were related 
to the space of derivations ${\rm Der}(V, W)$. 
 We find the following 
\begin{proposition}
\label{1st-coh}
Let $V$ be a grading-restricted vertex algebra and 
$W$ a $V$-module. Then  $H_{m}^{1}(V, W)$ 
is linearly isomorphic to the space $\mbox{\rm Der}\;(V, W)$ 
of derivations from $V$ to $W$ for any $m\in \Z_{+}$. 
\end{proposition}
In the case of a foliation, 
we have the following identifications in \eqref{onepointlocus}
\begin{eqnarray}
\mathcal G(\phi(p)) &=& \mathcal G(v, c(p)) = \Phi(v, c(p)), 
\nn
  f(\psi(p)) &=& \omega (v, c(p)),  
\nn
\phi(p) &=&(u, p), 
\nn
f( \psi(p')).\phi(p) &=&  \omega (v, c(p')-c(p))u,
\end{eqnarray}
and a multi-point holomorphic connection $\mathcal G$ on $\mathcal M/\F$, is a $\C$-linear map 
\[
\mathcal G: V^{\otimes^n} \to 
 \W_{z_1, \ldots, z_n}.   
\] 
Thus the multi-point holomorphic connection has the form 
 \begin{equation}
\label{locus00}
\sum\limits_{ q, q' \in M/\F} \Phi \left( \omega_V(v_q, c(q) - c(q')) u, q \right) = \omega_W (u, c(p')) \; 
\Phi \left(v, c(p) \right) + \omega_W(v, c(p)) \; \Phi \left(u, c(p') \right).   
\end{equation}
\begin{remark}
Due to Lemma \ref{nezu} and Proposition \ref{nezc}, the definition of the multi-point holomorphic connection 
on $\mathcal M/\F$ is canonical, 
i.e., it 
does not depend on the choice of $\U$ and coordinates on $\mathcal M/\F$ and $\U$. 
\end{remark}
%
The meaning of the name of a transversal holmophic connection \eqref{transa0}
 is clear when we consider elements of the space $\W_{z_1, \ldots, z_n}$ for $\mathcal M/\F$,   
\[
G(p, p')=\omega_W(u, c(p'))\; \mathcal G(v, c(p)) + \omega_W(u, c(p')) \; \mathcal G(u, c(p'))=0, 
\]
with formal parameters associated to local coordinates $c(p)$. In particular, when 
$\omega_W(u, t(p))$ is considered on a the transversal section, and 
$\mathcal G(v, l(p))$   
on a leaf of $\mathcal M/\F$, it relates objects on mutually transversal structures. 
This type of connections will appear in considerations of the second vertex algebra cohomology 
$H^2_{ex}(V, \W, \F)$ in Section \ref{coboundary}.

In what follows, to shortcut notations, we will denote by $p$ the origin of a local coordinate $c(p)$ at $p$, i.e., 
$c(p)|_p=0$. 
Let us introduce another 
\begin{definition}
A one fixed-point $p'$ holomorphic connection for the space \eqref{ourbicomplex11} is defined by 
 \begin{equation}
\label{onecone}
\sum\limits_{ q, q' \in M/\F} \Phi \left( \omega_V(v_q, c(q) - c(q')) u, q \right) = \omega_W (u, p') \; 
\Phi \left(v, c(p) \right) + \omega_W(v, c(p)) \; \Phi \left(u, p' \right).    
\end{equation}
\end{definition}
In particular, for the space $C^1_m(p'_r; V, \W, \F)$ we obtain 
\begin{equation}
\label{onecone1}
 \Phi \left( \omega_V(v, p' - c(p)) u, c(p)  \right) = \omega_W (u, p') \; 
\Phi \left(v, c(p) \right) + \omega_W(v, c(p)) \; \Phi \left(u, p' \right),.    
\end{equation}
We denote the space of such connections with a fixed point $p$ as ${\mathcal Con}_{p'}(m; V, \W)$. 
In Section \ref{coboundary} we have introduced the notion (Definition \ref{fixedpointcohomology}) of a fixed-point  
cohomology  
$H^n_m\left(p; V, \W,  \F \right)$. In particular, for $n=1$,  
\[
H^1_m \left(p'_r; V, \W,  \F \right) = {\rm Ker}\; \delta^1_m / {\rm Im} \; \delta^{0}_{m+1}|_{p'_r},  
\]
for a point $p'_r \in U_r$ of a transversal basis $\U$. 
%
The result of this section is in the following  
%
\begin{proposition}
\label{propres}
The vertex algebra first cohomologies $H_{m}^{1}(V, \W, \F)$, 
 $m \ge 0$ of a codimension one foliation $\F$  
are isomorphic to the space ${\mathcal Con}_{p'_r}(m; V, \W)$, for all $p'_r \in U_r$, $1 \le r \le m$,   
of holomorphic fixed point two 
point connections on the space of leaves $\mathcal M/\F$ with mappings     
composable with $m$ vertex operators on transversal sections. 
\end{proposition}
\begin{remark}
%
In contrast to the cohomologies $H_{m}^{1}(V, W)$ for a grading-restricted vertex algerba \cite{Huang}, 
 the cohomologies $H_{m}^{1}(V, W, \F)$ are not isomorphic for various $m$,  
since they contain dependence not on just formal parameters, but these formal parameters are identified with 
local coordinates around points on $\mathcal M$ on either the leaves of $\mathcal M/\F$ or transversal sections. 
 Indeed, connections $\mathcal G_m(v, z)$ are elements of the space $C^1_m(V, \W, \F)$, i.e., they 
are composable with $m$ vertex operators. 
\end{remark}
Now we proceed with the proof of Proposition \ref{propres}. 
\begin{proof}
Let us fix a point $p'_r$ with the local coordinate $t_r(p'_r)$ on the transversal section $U_r$   
 with origin at $p'_r$, i.e.,   
$t_r(p'_r)|_{p'_r}=0$. 
According to Proposition \ref{1st-coh} (cf. (1.1) in \cite{Hu3}), 
the cohomologies $H^1_m(V, W)$ of $V$ are given by the space of derivations. 
In terms of Definition \ref{fpconnection}, it coinsides with the space of 
fixed point holomorphic connections, i.e.,  ${\rm Der(V, W)}={\mathcal Con}_{p'_r}(V, W)$.   
Note that, 
for any 
\[
\Phi(v, c_r(p'_r)|_{p'_r}) \in C^1_m(p'_r; V, \W, \F),
\]  
such that 
\[
\langle w', \delta^1_m \Phi\left(v, p'_r \right) \rangle= \langle w', G_2\left(p'_r, p_2 \right) \rangle 
 =0,  
\]
i.e., 
\begin{eqnarray}
\label{probab1} 
0 &=& 
 \langle w',    
\omega_{W} \left(v_1, p'_r \right) \; \Phi \left(v_2, c_2(p_{2}) \right) \rangle  
\nn
&-&  
\langle w', \Phi \left( \omega_{V} (v_{1}, p'_r - c_2(p_{2})) v_{2}, c_2(p_2) \right) \rangle 
\nn 
\qquad \qquad   &+& 
\langle w', \omega_{W}( v_{2}, c_2(p_{2}) )  \; 
\Phi \left( v_{1}, p'_r \right) \rangle,   
\end{eqnarray}
i.e.,  \eqref{probab1} results in an element of the space 
${\mathcal Con}_{p'_r}(V, \W)$
 of one fixed point $p'_r$ holomorphic connections.    
In addition, by direct computation for any $\Phi' \in C^0_m(p'_r; V, \W, \F)$, we find 
\[
 \langle w', \delta^0_{m+1} \Phi' \rangle 
=  
\langle w', \omega_{V}(v, z) \; \Phi' \rangle  - \langle w', \omega_{V}(v, z) \; \Phi' \rangle
=0. 
\]
i.e., 
\[
{\rm Im} \; \delta^0_{m+1} \Phi'=\{0\}.
\] 

Conversely, for any element $g(v, 0)$ of ${\mathcal Con}_{p'_r}(V, \W)$, and $v \in V$,  
let us consider 
\begin{equation}
\label{lastlast}
\Phi_{g} = g(\omega_V(v, z)\one_V, p'_r)= \omega^{W}_{WV}( g(v, p'_r), z )\one_V, 
\end{equation} 
where we have used Lemma \ref{f-one}.
We had to express \eqref{lastlast} in terms of intertwining operator in order to show that 
\eqref{lastlast} is indeed composable with 
$m$ vertex operators and belong to the space $C^1_m(p'_r; V, \W, \F)$ with a fixed point $p'_r$.  
As it follows from 
\cite{FHL}, 
the map from $V$ to $\W_{z}$  given by 
\[
 v \mapsto \omega_{WV}^{W}( \Phi_g(v, p'_r), z_{1}) \one_V,  
\]
 is composable with 
$m$ vertex operators for any $m\in \N$.  
Thus $\Phi_{g} \in C_{m}^{1}(V, \W, \F)$ for any $m \in \N$. 
For $v_{1}$, $v_{2}\in V$, and $w'\in W'$, 
by using 
\eqref{wprop}, 
we find by direct computation 
\begin{eqnarray*}
&&
\langle w', \delta_{m}^{1} \Phi_g(v_{1}, c_1(p_1);  v_{2},  c_2(p_2))  
\nn
&&
=
\langle w', 
\omega_{W}(v_{1}, c_1(p_1) ) \; \omega_{WV}^{W}(g(v_{2}, p'_r), c_2(p_2))\one_V\rangle 
\nn
&&
\quad - 
\langle w', \omega_{WV}^{W}( g(\omega_{V}(v_{1}, c_1(p_1)-c_2(p_2))v_{2}, p'_r), c_2(p_2))\one_V)  
\rangle 
\nn
&&
\quad 
+ 
\langle w', \omega_{W}(v_{2}, c_2(p_2)) \; \omega_{WV}^{W}(g(v_{1}, p'_r), c_1(p_1))\one_V\rangle
\end{eqnarray*}
\begin{eqnarray*}
&&= 
\langle w', 
\omega_{W}(v_{1}, c_1(p_1))\; \omega_{WV}^{W}( g(v_{2}, p'_r), c_2(p_2))\one_V\rangle 
\nn
&&\quad - 
\langle w', \omega_{WV}^{W} (\omega_{WV}^{W}(g(v_{1}, p'_r), c_1(p_1) - c_2(p_2))v_{2}), c_2(p_2) )\one_V)
\rangle  
\nn
&&\quad - 
\langle w', \omega_{WV}^{W} (\omega_{W}(v_{1}, c_1(p_1)- c_2(p_2)) \; g(v_{2}, p'_r), c_2(p_2))\one) 
\rangle 
\nn
&&
\quad +  
\langle w', \omega_{W}(v_{2}, c_2(p_2)) \; \omega_{WV}^{W}(g(v_{1},p'_r), c_1(p_1))\one_V\rangle 
\end{eqnarray*}
\begin{eqnarray}
&&= 
\langle w',  
\omega_{W}(v_{1}, c_1(p_1)) \; \omega_{WV}^{W} (g(v_{2}, p'_r), c_2(p_2))\one_V\rangle 
\nn
&&\quad - 
\langle w', e^{c_2(p_2) L_{W}(-1)}
\omega_{WV}^{W} (g(v_{1}, p'_r), c_1(p_1)- c_2(p_2))v_{2} 
\rangle 
\nn
&&\quad - 
\langle w', e^{ c_2(p_2) L_{W}(-1)}\omega_{W}(v_{1}, c_1(p_1)-c_2(p_2))g(v_{2}, p'_r) 
\rangle 
\nn
&&
\quad +  
\langle w', \omega_{W}(v_{2}, c_2(p_2)) \; \omega_{WV}^{W}(g(v_{1}, p'_r),
 c_1(p_1))\one_V\rangle  
\nn 
&&
\nn
&&= - 
\langle w', \omega_{WV}^{W}(g(v_{1}, p'_r), c_1(p_1)) \omega_{V}(v_{2}, c_2(p_2))\one_V\rangle 
\nn
\label{d-phi-f}
&&\quad
+ 
\langle w', \omega_{W}(v_{2}, c_2(p_2))\; \omega_{WV}^{W}(g(v_{1}, p'_r), c_1(p_1))\one_V\rangle.  
\nn
\end{eqnarray}
By using Theorem 5.6.2 in \cite{FHL}, we derive that 
\eqref{d-phi-f} vanishes. 
Therefore we obtain a linear map 
\[
g(v, p'_r) \mapsto \Phi_{g}, 
\]
 from the space 
\[
{\mathcal Con}_{p'_r}(V, \W) =\mbox{\rm Der}\;(V, \W) \to 
H_{m}^{1}(V, \W)=C_{m}^{1}(p'_r; V, \W).  
\]
Thus we find, that 
\begin{equation}
\label{partal}
H^1_m(p'_r; V, \W, \F) = {\mathcal Con}_{p'_r}(V, \W). 
\end{equation}
By moving $p'_r \in U_r$ all along $t_r(p'_r)$ we exhaust to all points at $U_r$, 
we obtain connections of ${\mathcal Con}_{U_r}(V, \W)$ on the whole $U_r$.  
By using Lemma \ref{complexlemma}, we 
extend \eqref{partal} to 
we obtain the statement of Proposition:
\[
H^1_m(V, \W, \F) = \bigcup\limits_{p_r \in U_r}{\mathcal Con}_{p'_r}(V, \W).    
\]
\end{proof}
\section{Cohomological classes}
\label{cohomological}
In this section we describe certain classes associated to the first and the second vertex algebra cohomologies
 for codimension one foliations.   
Usually, the cohomology classes for codimension one foliations \cite{G, CM, Ko} are introduced  by means of  
an extra condition (in particular, the orthogonality condition) applied to differential forms, and leading to 
the integrability condition. 
As we mentioned in Section \ref{coboundary}, it is a separate problem to introduce a product defined on one 
or among various spaces 
  $C^n_m(V, \W, \F)$ of \eqref{ourbicomplex}.  
Note that elements of $\mathcal E$ in \eqref{deltaproduct} and $\mathcal E_{ex}$ in \eqref{mathe2} can be seen as elements 
of spaces $C^1_\infty(V, \W, \F)$, i.e., maps composable with an infinite number of vertex operators. 
Though the actions of coboundary operators $\delta^n_m$ and $\delta^2_{ex}$ in \eqref{deltaproduct} and 
\eqref{halfdelta1} are written in form of a product 
 (as in Frobenius theorem \cite{G}), and, in contrast to the case of differential forms, 
 it is complicated to use these products 
 for further formulation 
of cohomological invariants and derivation of analogues of the Godbillon-Vey invariants.   
Nevertheless, even with such a product yet missing, it is possible to introduce the lower-level cohomological 
classes of the form $\left[ \delta \eta \right]$ which are counterparts of the Godbillon class \cite{Galaev}.    
Let us give some further definitions. 
By analogy with differential forms, let us introduce
\begin{definition}
\label{probab}
 We call a map 
\[
\Phi \in C_{k}^{n}(V, W, \F),  
\] 
closed if it is a closed connection: 
\[
\delta^{n}_{k} \Phi=G(\Phi)=0.
\]  
For $k \ge 1$, we call it exact if there exists 
$\Psi \in  C_{k-1}^{n+1}(V, W, \F)$  
such that $\Psi=\delta^{n}_{k} \Phi$, i.e., $\Psi$ is a form of connection.  
\end{definition}
For $\Phi \in {C}^{n}_{k}(V, W, \F)$ we call the cohomology class of mappings 
 $\left[ \Phi \right]$ the set of all closed forms that differ from $\Phi$ by an 
exact mapping, i.e., for $\chi \in {C}^{n-1}_{k+1}(V, W, \F)$,  
\[
\left[ \Phi \right]= \Phi + \delta^{n-1}_{k+1} \chi. 
\]

As we will see in this section, there are cohomological classes, 
(i.e., $\left[\Phi \right]$, $\Phi \in {C}^{1}_{m}(V, W, \F)$, $m\ge 0$),  
 associated with two-point connections and the first cohomology ${H}^{1}_{m}(V, W, \F)$, and classes  
(i.e., $\left[\Phi \right]$, $\Phi \in {C}^{2}_{ex}(V, W, \F)$),  
associated with transversal connections and the second cohomology ${H}^{2}_{ex}(V, W, \F)$, 
of $\mathcal M/\F$. 
The cohomological classes we obtain are vertex algebra cohomology counterparts of the Godbillon class 
\cite{Ko, Galaev}   
for codimension one foliations.  
\begin{remark}
As it was discovered in \cite{BG, BGG}, 
 it is a usual situation when the existence of a connection (affine of projective) 
for codimension one foliations on smooth manifolds prevents corresponding cohomology classes from vanishing.  
 Note also, that for a few examples of codimension one foliations, the cohomology class $\left[ d\eta \right]$  
is always zero.  
\end{remark}
\begin{remark}
In contrast to \cite{BG}, our cohomological class is a functional of $v\in V$. 
That means that the actual functional form of $\Phi(v, z)$ (and therefore $\langle w', \Phi\rangle$, for 
$w'\in W'$)    
 varies with various choices of $v\in V$. 
That allows one to use it in order to distinguish types of leaves of $\mathcal M/\F$. 
\end{remark}
\subsection{Classes associated with the first cohomologies $H^1_m(V, \W, \F)$}
For the first cohomology $H^1_m(V, \W, \F)$, we have the following corollary from Proposition \ref{propres}:   
\begin{corollary}
The $H^1_m(V, \W, \F)$ cohomological class of the grading-restricted vertex algebra cohomology of the leaf space 
$\mathcal M/\F$ is given by 
\begin{equation}
\label{tsfat}
\left[ \delta^1_m \Phi\right], 
\end{equation}
 for $\Phi\in C^1_m$. 
It's vanishing if and only if $\Phi$ is given by a two 
 point holomorphic connection. 
\end{corollary}
\begin{remark}
Non-vanishing cohomological invariants of the form \eqref{tsfat} are used in Section \ref{characterization} in order to 
characterize leaves of $\mathcal M/\F$ and transversal sections. 
\end{remark}
\begin{proof}
$\left[ \delta^1_m \Phi\right]$ for $\Phi\in C^1_m$.
It is easy to see that it remains cohomologically invariant under a substitution 
\[
\Phi \mapsto \Phi + \Phi_0, 
\]
due to properties of \eqref{tsfat}. 
The second statement of the proposition 
 follows from the proof of Proposition \ref{propres}. 
 In Subsection \ref{bordo} we will explain which role the cohomological invariant \eqref{tsfat} for foliation $\F$.
\end{proof}
\subsection{Classes associated with exceptional cohomology} 
In this subsection we consider the exceptional cohomology $H^2_{ex}(V, \W, \F)$ associated to the short complex 
 \eqref{hat-complex-half}, 
and corresponding cohomological class. 
Let us first recall some definitions \cite{Hu3}  concerning the notion of square-zero extension 
of $V$ by its module $W$ which is an analogue of the notion of 
square-zero extension of an associative algebra by a bimodule (cf. 
\cite{W}).
\begin{definition}
 Let $V$ be a grading-restricted vertex algebra.  
 A square-zero ideal of $V$ is an ideal $W$ of $V$ such that 
for any $u$, $v\in W$,
\[
Y_{V}(u, x)v=0.
\] 
\end{definition}
\begin{definition}
 Let $V$ be a grading-restricted vertex algebra and $W$ 
a $\Z$-graded $V$-module. 
A  square-zero extension $(V_W, \gamma, \alpha)$ of $V$ by $W$ 
is a grading-restricted vertex algebra $V_W$   
together with a surjective homomorphism  
\[
\gamma: V_W \to V, 
\] 
of grading-restricted vertex algebras such that $\ker \gamma$ 
is a square-zero ideal of $V_W$ (and therefore a $V$-module)
and an injective homomorphism $\alpha$ of $V$-modules from $W$ to $V_W$ 
such that 
\[
\alpha(W)=\ker \gamma.
\] 
\end{definition}
\begin{definition}
Two square-zero extensions $(V_{W,1}, \gamma_{1}, \alpha_{1})$
 and 
$(V_{W,2}, \gamma_{2}, \alpha_{2})$ of $V$ by $W$ are equivalent if there 
exists an isomorphism of grading-restricted vertex algebras
$h: V_{W, 1}\to V_{W, 2}$ 
such that the diagram
\[
\begin{CD}
0@>>> W @>>\alpha_{1}> V_{W, 1} @>>\gamma_{1}>V @>>>0\\
@. @V{\rm Id}_{W}VV @VhVV @VV{\rm Id}_{V}V\\
0@>>> W @>>\alpha_{2}> V_{W,2} @>>\gamma_{2}>V @>>>0,
\end{CD}
\]
is commutative.
\end{definition}
Let $(V_W, \gamma, \alpha)$ be a square-zero extension of $V$ by $W$. 
It is possible to construct a realization of the square-zero extension of $V$ by $W$ on $Z=V\bigoplus W$. 
Then there exists an injective linear map  $\Gamma: V\to V_W$, 
such that the linear map 
\[
h: Z \to V_W, 
\]
 given by
\[
h(v, w)=\Gamma(v)+ \alpha(w), 
\]
 is a linear isomorphism. 
By definition,
the restriction of $h$ to $W$ is the isomorphism $\alpha$
from $W$ to $\ker \gamma$.
 Then the grading-restricted vertex algebra structure
and the $V$-module structure on $V_W$  
give a grading-restricted vertex algebra structure and a
$V$-module structure on $Z$ such that the embedding 
$i_{2}: W\to Z$  and 
the projection $p_{1}: Z \to V$, 
are  homomorphisms of grading-restricted vertex algebras. 
In addition to that, $\ker p_{1}$ is a square-zero ideal of $Z$, $i_{2}$ is an injective homomorphism
such that $i_{2}(W)=\ker p_{1}$ and
the diagram
\begin{equation}
\begin{CD}
0@>>> W @>i_{2}>> Z @>p_{1}>>V @>>>0\\
@. @V{\rm Id}_{W}VV @VhVV @VV{\rm Id}_{V}V\\
0@>>> W @>>\alpha> V_W @>>\gamma>V @>>>0 
\end{CD}
\end{equation}
of $V$-modules 
is commutative. 
Thus one obtains a  square-zero extension $(Z, p_{1}, i_{2})$ 
equivalent to $(V_W, \gamma, \alpha)$.
It is enough then to consider square-zero extensions
of $V$ by $W$ of the particular form $(Z, p_{1}, i_{2})$.   
 The difference between  
two such square-zero extensions consists in the vertex operator maps. 
Such square-zero extensions will be denotesd by $(Z, Y_{Z}, p_{1}, i_{2})$.   

Let us first mention the geometrical meaning of the square-zero extension $(V_W, \gamma, \alpha)$ of $V$ by $W$.    
Let us consider $u$, $v$ belong to the square-zero ideal of a grading-restricted vertex algebra $V$, then 
\[
\omega_{V} (u, c(p))v=0.  
\]
Then, geometrically it means that corresponding vertex algebra holomorphic connections are 
transversal (cf. Definition \ref{transcon}): 
\begin{equation}
\label{tranzo}
G_{tr}(p,p')=\omega_W \left(v, c(p') \right) \Phi\left(u, c(p)\right) 
+ \omega_W \left(u, c(p) \right) \; \Phi \left(\psi, c(p')\right)=0.  
\end{equation}
Note that, for a square-zero ideal,  the full form of holomorphic connection has a reduced form 
\eqref{tranzo}. 
In \cite{BG, BGG} it was shown that certain cohomological class vanishes if and only if 
the exist an affine or projective connection. 
In our setup the holomorphic connection plays a similar role: if it has does not have an full closed form
 \eqref{locus00}, then 
the cohomology class it non-trivial. 

In \cite{Hu3} we find the proof of the following algebraic result for the second cohomology of a grading-restricted 
vertex algebra $V$, 
$H^{2}_{ex}(V, W)$ of $V$ with coefficients in $W$. 
It follows from that Proposition, that  
the difference between 
two square-zero extensions are controled by 
 the vertex operator map for the square-zero extension defined for $Z=V\bigoplus W$.  
\begin{proposition}
\label{h2-ext}
Let $V$ be a grading-restricted vertex algebra and $W$ 
a $V$-module.
 Then the set of the equivalence classes of 
square-zero extensions of $V$ by $W$ corresponds bijectively to $H^{2}_{ex}(V, W)$.
\end{proposition}
Now we formulate the following corollary from Proposition \ref{h2-ext}.  
\begin{corollary}
\label{h2-ext}
Let $V$ be a grading-restricted vertex algebra and $W$ 
a $V$-module. 
The classes of square-zero extensions of $V$ by $W$ are isomorphic to 
  classes of 
cohomological invariants $
 \Phi$ \eqref{ef} of $H^{2}_{ex}(V, \W, \F)$. 
\end{corollary}
\begin{proof}
Now let us consider the cohomology $H^{2}_{ex}(V, \W, \F)$. 
Here, for $\Phi \in C^{2}_{ex}(V, \W, \F)$, 
 the kernel of $\delta^{2}_{ex}$ 
has the form of closed three-variable connection: 
\begin{eqnarray*}
&&
0= \delta^{2}_{ex} \Phi=  \mathcal G(p_1, p_2, p_3) 
\nn
&&  
= 
  \langle w', \omega_{W} ( v_{1}, c_1(p_1) ) 
\;  
  \Phi  \left( v_{2}, c_2(p_2);  v_{3}, c_3(p_3) \right)  \rangle 
\nn
&&\quad 
- \langle w',
\Phi \left(   
\omega_V (v_{1}, c_1(p_1)) \; v_{2}, c_2(p_2); 
 v_{3}, c_3(p_3) \right) \rangle 
\nn
&&
\quad \quad 
+ 
\langle w',  \Phi \left( v_{1}, c_1(p_1); 
\omega_V ( v_{2}, c_2(p_2) ) \;v_{3}, c_3(p_3) 
 \right)  
 \rangle 
\nn
 &&\quad \quad 
- \langle w',  
\omega_{W} (v_{3}, c_3(p_3) )  \; \Phi \left( v_{1}, c_1(p_1); v_{2}, c_2(p_2) \right)   
\rangle, 
\end{eqnarray*}
for $w'\in W'$,
$\Phi\in {C}_{ex}^{2}(V, \W, \F)$,
$v_{1}, v_{2}, v_{3}\in V$,
 and $(z_{1}, z_{2}, z_{3})\in F_{3}\C$.   
For $\Phi' \in C^1_2(V, \W, \F)$, 
 the image of $\delta^1_2$, 
 has a non-closed connection form: 
\begin{eqnarray}
 &&
{\delta}^{1}_{2} \Phi' = G_2(p_1, p_2)  
\nn
 && \qquad = 
\langle w', \omega_{W} \left(  v_{1}, c_1(p_1) \right)  \;  
\Phi' \left(  v_{2}, c_2(p_2) \right) \rangle 
\nn
&&
\qquad -  
\langle w',  
\Phi' \left(  \omega_{V} \left( v_{1}, c_1(p_1) \right) v_{2}, c_2(p_2) \right) \rangle 
\nn
 &&\quad \quad 
+  
\langle w',  
 \omega_W \left(v_{2}, c_2(p_2) \right) \; \Phi' \left(v_{1}, c_1(p_1) \right)     
\rangle,  
\end{eqnarray}
for $w'\in W'$,
$\Phi\in {C}_{2}^{1}(V, \W, \F)$, 
$v_{1}, v_{2} \in V$ and $(z_{1}, z_{2})\in F_{2}\C$. 
 The cohomology $H^2_{ex}(V, \W, \F)$ is given by the factor-space of closed three-variable
 connections over non-closed two-variable connections. 
We will show now that this 
factor space is actually parameterized by the vertex operator maps for a square-zero extension of 
$V$ by $W$. 

The explicit definition for $Z$-vertex operator was introduced in \cite{Hu3}.   
We denote by $(Z, Y_Z, p_{1}, i_{2})$ a suitable square-zero extension of $V$ by $W$.
Then there exists 
\[
\omega_\Psi\left(
 u, z\right)v \; \in \W((z)), 
\]
 for $u$, $v\in V$ such that 
\begin{eqnarray*}
\omega_{Z}((v_{1}, 0), z)(v_{2}, 0)&=&(\omega_{V}(v_{1}, z)v_{2}, \omega_\Psi(v_{1}, z)v_{2}),
\\
\omega_{Z}((v_{1}, 0), z)(0, w) & =& (0, \omega_{V}(v_{1}, z)w_{2}),
\\
\omega_{Z}((0, w_{1}), z)(v_{2}, 0)&=&(0, \omega^{W}_{WV}(w, z)v_{2}),
\\
\omega_{Z}((0, w_{1}), z)(0, w_{2})&=& 0,   
\end{eqnarray*}
for $v_{1}$, $v_{2}\in V$ and $w_{1}$, $w_{2}\in W$. 
Thus one has 
\begin{eqnarray}
\label{Y_V+W}
&&\omega_{Z}((v_{1}, w_{1}), z)(v_{2}, w_{2}) 
\\
&& \quad =\left(\omega_{V}(v_{1}, z)v_{2},  
\omega_{W}(v_{1}, z) \; w_{2} 
 + \omega_{WV}^{W}(w_{1}, z)\;v_{2}   + \omega_\Psi(v_{1}, z)\;v_{2}\right), 
\nonumber
\end{eqnarray}
for $v_{1}, v_{2}\in V$ and $w_{1}, w_{2}\in W$. 
The vacuum of $Z$ is given by  $(\one_V, 0)$, and 
\[
\omega_\Psi(v, z)\one_V=0,
\]
 and 
the dual space $Z'$ for $Z$ is identified with 
\[
Z'=V'\oplus W'. 
\] 
By Definition \ref{grares} of a grading-restricted vertex algebra, 
for $v$, $v'\in V$, 
vertex operators $\omega_\Psi(v, z)$ and  $\omega_{V}(v', z')$ in extension $(V_W, \gamma, \alpha)$,
satisfy the associativity property, i.e., 
their matrix 
elements of \eqref{porosyata}  
converge (under appropriate conditions for local coordinates of points)
 to the same $\W_{z_{1}, z_{2}}$-valued rational function. 
Thus, for $v_{1}, v_{2}\in V$, and $(z_1, z_2) \in F_2\C$, we introduce a linear map 
\[
\Phi: V\otimes V\to \W_{z_{1}, z_{2}}, 
\]
\begin{eqnarray}
\label{ef}
\Phi(v_1, z_1; v_2, z_2)&=&  E(\omega_\Psi(v_{1}, z_{1})\; \omega_{V}(v_{2}, z_{2})\one_V)
\nn
&=&  E(\omega_\Psi(v_{2}, z_{2}) \; \omega_{V}(v_{1}, z_{1})\one_V) 
\nn
&=& E(\omega_{WV}^{W}(\omega_\Psi(v_{1}, z_{1}-z_{2}) v_{2}, z_{2})\one_V). 
\end{eqnarray} 
As in the proof of Proposition \ref{h2-ext} we check that $\Phi$ \eqref{ef} satisfies the $L(-1)$-derivative and $L(0)$- 
conjugation properties.
Since $Z$ is a grading-restricted vertex algebra, 
 by using the associativity property for vertex operators \eqref{Y_V+W},
we see that the conditions \eqref{pervayaforma} and \eqref{vtorayaforma} for forms $G_i$, $i=1$, $2$, in the proof of 
 Lemma \ref{lemmo} of the space  $C^2_{ex}(V, \W, \F)$  for $\Phi$ are satisfied, 
and $\Phi \in {C}_{ex}^{2}(V, \W, \F)$.   
Using again corresponding associativity properties 
for vertex operators in $Z$, we find that 
 $\Phi$ is closed (according to our Definition \ref{probab}),   
i.e., 
\[
 \delta_{ex}^{2}\Phi=0.
\] 
Thus, we see that, for a representative of the class of square-zero extension 
 $(Z, Y_Z, p_{1}, i_{2})$  
corresponds by the formula \eqref{ef} for $\omega_{Z}$ to an element 
of $H_{ex}^{2}(V, \W, \F)$, 
\[
\left[ \Phi \right] = \Phi + \eta, 
\]
where $\eta$ be an element $\delta_{2}^{1} C_{2}^{1}(V, \W, \F)$. 
It is easy to see that, according to properties of the above construction
$\Phi$ is invariant with respect to a substitution 
\[
\Phi \mapsto \Phi + \mu, 
\]
for $\mu \in C^{2}_{ex}(V, \W, \F)$. 
Thus, $\Phi$ \eqref{ef} belongs to the cohomology class $H^2_{ex}(V, \W, \F)$.   

Let us prove the inverse statment. 
  For an element $\Phi \in C^{2}_{ex}(V, \W, \F)$ which is a representative 
of $H^{2}_{ex}(V, \W, \F)$, 
 according to Definition \ref{composabilitydef} of composibility, 
 it follows that for any $v_{1}$, $v_{2}\in V$, 
there exists $N^{2}_{0}(v_1, 0)$ such that  
for $w'\in W'$, 
\[
G_2(c_1(p_1), c_2(p_2))=\langle w', \Phi(v_{1}, z_1; v_{2}, z_2)\rangle,  
\]
is a rational $\W_{z_1, z_2}$-valued form with the only possible pole at
$z_1=z_2$ of order less than or equal to $N^2_0(v_1, v_2)$. 
For $v_{1}$, $v_{2}\in V$, 
let us define $\omega_\Psi(v_{1}, \zeta)v_{2}\in \W((\zeta))$   
such that 
\[
 \langle w', \omega_\Psi(v_{1}, \zeta)v_{2}\rangle|_{\zeta=z}
=\langle w', \Phi(v_{1}, z; v_{2}, z_2)\rangle,  
\]
for $z\in \C^{\times}$. 
For $v_{1}$, $v_{2}\in V$, we can define 
$Y_{Z}(v_{1}, \zeta)v_{2}$ using (\ref{Y_V+W}). 
Thus, we obtain a 
vertex operator map $Y_{Z}$, and 
$Z$ is endowed with the structure of a grading-restricted 
vertex algebra. 
Finally, we have   
\begin{corollary}
Two elements of $\ker \delta_{ex}^{2}$ 
differ by an element $\delta^{1}_2 C^{1}_2 (V, \W, \F)$ 
 if and only if  
the corresponding square-zero extensions of $V$ by $W$ are
equivalent.
\end{corollary}
\end{proof}
\section{Characterization of leaves and transversal sections of foliations on complex curves}   
\label{characterization}
In this section we consider a general formulation of characterizasion of $\mathcal M/\F$ in codimension one case 
by means of rational functions of invariants. 
Let us introduce further notations, for $n \ge 0$, 
\[
{\bf x}=(x_1, \ldots, x_n), 
\]
for $n$ vertex algebra element, formal parameters, points, etc. 
Introduce 
\begin{definition}
\label{torboid}
For an element $\Phi({\bf v}, {\bf c(p)}) \in \W_{c_1(p_1), \ldots, c_n(p_n)}$ let us call  
 $n$-variable rational function valued form 
\begin{equation}
\label{torba}
\mathcal R({\bf z})= \langle w', \Phi({\bf v}, {\bf c(p)}) \rangle, 
\end{equation}
the characteristic form.    
\end{definition}
We have used this form for the construction of chain complexes in Section \ref{coboundary}. 
In certain cases, depending on properties of $F({\bf v}, {\bf z})$, one is able to compute this matrix element 
explicitely.
By varying vertex algebra elements $v_i$, one can vary the the form of 
dependence of $\Phi({\bf v}, {\bf c(p)})$ on ${\bf v}$, 
and, therefore, obtain various functions of $R({\bf z})$.  
 By using the freedom of choice of $v\in V$, we could try to find 
 a suitable pattern for of $\Phi({\bf v}, {\bf c(p)})$ (as a functional of $v$), 
 in such a way \eqref{torba} would result to a specific differential form. 
Since $\Phi({\bf v}, {\bf c(p)})$ belongs to $C^n_m(V, \W, \F)$ for some $n$, $m$, 
it is important to mention that,  due to our formulation in terms of matrix elements, 
 \eqref{torba}, associated to cohomological invariants are supposed to be absolutely convergent 
in suitable domains of $\mathcal M/\F$. 
Depending on analytical properties with respect to local coordinates on $c(p)$ 
 of \eqref{torba} one can use it in order to characterize or distinguish particular 
leaves and transversal sections on $\mathcal M/\F$.  
For that purpose one can also integrate \eqref{torba} along (closed) paths either on a leaf of $\mathcal M/\F$ or on 
a transversal section of $\U$.
For that purpose we introduce 
\begin{definition}
We call a multiple integral 
\begin{equation}
\label{integral}
F({\bf z'})=\int\limits_{(p_1)}^{(p_2)} \mathcal R\left({\bf c(p)} \right),   
\end{equation}
the characteristic function for $\mathcal M/\F$, where $(p_i)$, $i=1, 2$ denote limiting points of integration.   
\end{definition}
The idea of integration of $\mathcal R\left({\bf c(p)}\right)$ goes back to \cite{Thur}. 
In Proposition \ref{nezc} we proved, in particular,  that elements of 
spaces $C^n_m(\V, \W, \F) \in \W_{z_1, \ldots, z_n}$ 
are invariant with respect to changes of formal parameters $(z_1, \ldots, z_n)$.
In Definition \ref{torboid} of a characteristic form we use such elements, and, therefore, \eqref{torba}, 
containing $\wt (v_i)$, $1 \le i \le n$,  
of corresponding differentials, is also invariant with respect 
to action of $\left({\rm Aut}\; \Oo^{(1)}\right)^{\times n}$.
Below we enumerate $\W_{z_1, \ldots, z_n}$-elements suitable for such characterization. 
\subsection{Composibilty condition}
Let us start with forms associated to the composibility conditions. 
 For  
 $l_{1}, \dots, l_{n}\in \Z_+$ such that $l_{1}+\cdots +l_{n}=n+m$, 
define ${k_1}={l_{1}+\cdots +l_{i-1}+1}$, ..., ${k_i}={l_{1}+\cdots +l_{i-1}+l_{i}}$. 
Consider a set of $p_{k_1}, \ldots, p_{k_n}$ with local coordinates $c_{k_1}(p_{k_1}), \ldots, c_{k_n}(p_{k_n})$, 
on $\mathcal M$ for  
 points on $\mathcal M/\F$. 
Then, for 
$v_{1}, \dots, v_{n+m}\in V$ and $w'\in W'$, 
one defines  \eqref{psii} 
and there exist
  positive integers $N^n_m(v_{i}, v_{j})$
depending only on $v_{i}$ and $v_{j}$ for $i, j=1, \dots, k$, $i\ne j$ such that the series \eqref{Inm} 
 is absolutely convergent  when for $l_p=l_{1}+\cdots +l_{i-1}+p$, $l_q=l_{1}+\cdots +l_{j-1}+q$, 
\begin{equation}
\label{condi1}
|c_{l_p}(p_{l_p})-\zeta_{i}| 
+ |c_{l_q}(p_{l_q})-\zeta_{i}|< |\zeta_{i}
-\zeta_{j}|, 
\end{equation}
for $i,j=1, \dots, k$, $i\ne j$ and for $p=1, 
\dots,  l_i$ and $q=1, \dots, l_j$. 
Note that in \eqref{Inm} the original variables $z_i$ are present in combinations \eqref{psii} only, 
and the conditions 
on domains of convergence are express through such combinations 
$c_{l_p}(p_{l_p})$ and $c_{l_q}(p_{l_q})$, and some $\zeta_i$ which could be identified with 
other local coordinates on $\mathcal M$ for $\mathcal M/\F$. 
Thus we obtain an external (with respect to original coordinates) condition on $\mathcal I^n_m(\Phi)$.  
Geometrically this means that the sum of shifts in domains of convergency  
 with respect to $c_{l_p}(p_{l_p})$ and $c_{l_q}(p_{l_q})$ 
are smaller than difference for other two points with local coordinates $\zeta_i$ and $\zeta_j$. 
It is also assumed that 
the sum must be analytically extended to a
rational function
in $(c_{1}(p_1), \dots, c_{m+n}(p_{m+n}))$,
 independent of $(\zeta_{1}, \dots, \zeta_{n})$,  
with the only possible poles at 
$c_{i}(p_i)=c_{j}(p_j)$, of order less than or equal to 
$N^n_m(v_{i}, v_{j})$, for $i,j=1, \dots, k$,  $i\ne j$. 

Consider the second condition in Definition \ref{composabilitydef}. 
For $v_{1}, \dots, v_{m+n}\in V$, there exist 
positive integers $N^n_m(v_{i}, v_{j})$, depending only on $v_{i}$ and 
$v_{j}$, for $i, j=1, \dots, k$, $i\ne j$, such that for $w'\in W'$, 
 such that \eqref{Jnm}
 is absolutely convergent when $z_{i}\ne z_{j}$, $i\ne j$
\begin{equation}
\label{condi2}
|c_{i}(p_i)|>|c_{k}(p_k)|>0, 
\end{equation}
 for $i=1, \dots, m$, and 
$k=m+1, \dots, m+n$, and the sum can be analytically extended to a
rational function 
in $(z_{1}, \dots, z_{m+n})$ with the only possible poles at 
$z_{i}=z_{j}$, of orders less than or equal to 
$N^n_m(v_{i}, v_{j})$, for $i, j=1, \dots, k$, $i\ne j$,. 
Elements $\Phi$ of spaces $C^n_m(V, \W, \F)$ \eqref{ourbicomplex} are composable with $m$ vertex operators, and, 
therefore possess properies described above.  
Due to absulute convergence in the regions \eqref{condi1} and \eqref{condi2} on $\mathcal M/\F$, forms $I^n_m(\Phi)$ and  
$J^n_m(\Phi)$ 
locally characterize $\mathcal M/\F$. 
\subsection{The cohomological class $\left[ \delta^1_m \Phi \right]$ of $H^1_m(V, \W, \F)$} 
\label{bordo}
In Section \ref{spaces} we have proved that for $\Phi \in C^1_m (V, \W, \F)$, the invariant $\delta^1_m \Phi$
vanishes if and only if $\Phi$ is a one fixed point holomorphic connection. 
Here, since $\Phi \in C^1_{m} (V, \W, \F)$ and $\delta^1_m \Phi \in C^2_{m-1} (V, \W, \F)$, we have 
the  characteristic two-form for $\mathcal M/\F$
\begin{equation}
\label{twoform}
\mathcal R(c(p), c(p')) = 
\langle w', \left( \delta^1_m \Phi \right) \left(v, c(p); v', c(p') \right) \rangle.    
\end{equation}
%
In Section \ref{cohomological} we proved that $\delta^1_m \Phi$ represents a cohomological class of $H^1_m(V, \W, \F)$. 
For a characterization of leaves of $\mathcal M/\F$ we may choose instead elements 
\[
\Phi_g=g(v, 0) \in \W, 
\]
which do not depend on $z$, and, hence, then matrix elements become computable. 
For non-vanishing invariants \eqref{tsfat} (i.e., not two
 point connection valued $G(\Phi)$) we obtain 
 the non-vanishing form 
\begin{eqnarray}
\label{oneform}
\mathcal R(c(p))&=& 
\langle w', \delta^1_m \Phi_{g}\left(v, c(p) \right) \rangle  
\nn
&=& 
\langle w',   \omega_W \left(u, c(p) \right)\; g(v, 0) 
+ e^{zL_{W}(-1)} \omega_W \left(v, -c(p) \right) \; g(u, 0) 
\nn
&& \qquad \qquad \qquad \qquad \qquad \qquad \qquad \qquad 
- 
g( \omega_V(u, c(p) ) v, 0)  
  \rangle.    
\end{eqnarray}
The form of the dependence of $\Phi$ or $g(v, z)$ on $v\in V$ determines the result of 
taking the matrix element in \eqref{oneform}. 
In order to compute \eqref{oneform} we use the properties of the graiding-restricted vertex algebra $V$, 
 in particular, expand $\omega(v, c(p))$  
as in \eqref{vop}, and act on $g(v, 0)$.  
Recall that by construction of Section \ref{spaces}, $c(p)$ can be associated to either a local coordinate $l(p)$ of $p$ 
 on $\mathcal M$  
considered on a leaf of $\mathcal M/\F$ or a local coordinate $t(p)$ on a transversal section $U \in \U$. 
\subsection{The cohomological class $\left[\Phi \right]$ of $H^2_{ex}(V, \W, \F)$}
Recall definitions of the forms $G_1$ \eqref{pervayaforma} and $G_2$ \eqref{vtorayaforma} from Section \ref{coboundary}. 
We define the following characteristic functions as triple integrals associated to the these forms: 
\begin{equation}
\label{pervfu}
F(c(p), c(p'), c(p'')) =  \int\limits_{(q_1, q'_1, q''_1) }^{(q_2, q'_2, q''_2) } 
 G_i \left(  v, c(p); v', c(p'); v'', c(p'')\right), 
\end{equation}
with $i=1$, $2$. 
By assumption contatining in Subsection \ref{comtrsec}, the forms \eqref{pervayaforma} and  \eqref{vtorayaforma} have 
nice convergency properties. Moreover, they contain only parts of the connection (functions do not vanish), 
and can be used in order to describe leaves or sections of $\mathcal M/\F$.  
For the invariant related to the second cohomology $H^2_{ex}(V, \W, \F)$, 
we obtain for \eqref{ef}
\begin{equation}
\label{threeform}
F(c(p), c(p'), c(p'')) = 
\langle w', \Phi  \left(v, c(p); v', c(p'); v'', c(p'') \right) \rangle. 
\end{equation}
In addition to \eqref{threeform}, 
one uses the particular form of forms $G_i$, $i=1$, $2$ 
\begin{eqnarray*}
G_1(p_1, p_2, p_3) &=& \langle w', \omega_\Psi(v_{1}, c_{1}(p_1))\; 
\omega_{V} \left(v_{2}, c_{2}(p_2) \right)\; \omega_{V}(v_{3}, c_{3}(p_3)) \one_V \rangle
\nn
&& \quad + \langle w', \omega_{W}(v_{1}, c_{1}(p_1)) \; 
\omega_\Psi(v_{2}, c_{2}(p_2))\; \omega_{V}(v_{3}, c_{3}(p_3))\one_V\rangle,  
\end{eqnarray*} 
and 
\begin{eqnarray*}
G_2(p_1, p_2, p_3)  &=& \langle w', \omega_\Psi(\omega_{V}(v_{1}, c_{1}(p_1) - c_{2}(p_2))v_{2},  
c_{2}(p_2))\; \omega_{V}(v_{3}, c_{3}(p_3))\one_V\rangle  
\nn
&&\quad 
+\langle w', \omega_{WV}^{W}(\omega_\Psi(v_{1}, c_{1}(p_1) -c_{2}(p_2))v_{2}, 
c_{2}(p_2))\; \omega_{V}(v_{3}, c_{3}(p_3))\one_V\rangle,  
\end{eqnarray*}
\eqref{pervayaforma} and \eqref{vtorayaforma} in \eqref{pervfu} (cf. Subsection \ref{comtrsec}). 
In particular, these invariants allow to show the transversality of cycles for 
foliations defined by the vanishing real part
\[
{\rm Re} \; \Omega=0,  
\]
 of one form $\Omega$ on compact Riemann surface \cite{DNF, N0, N1, N2},  
in the hyperelliptic case, 
\[
w^2=P_{2g+2}(z)=\prod_j (z-z_j), \quad z_j\neq z_l, 
\]
where $P_{2n+1}(z)$ is a polynomial.  
It would be interesting to find a way how to distinguish non-differomorphic Reeb components for 
foliations of the torus \cite{Lawson, Thur, BGG}. These questions will be addressed in elsewhere. 

The crucial question is how one could distinguish (locally and globally)
 leaves and transversal sections of a foliation. 
In particular, we should be able to describe singular points (such as, e.g., saddle points for 
foliations on Riemann surfaces), one-point leaves, transversal cycles, non-diffeomorphic, 
 compact and non-compact leaves.   
In our construction, for $\Phi \in C^n_m(V, \W, \F)$, $n$ and $m$ can be associated to 
either corresponding number of points on leaves and transversal sections. 
 $\Phi\in \W_{z_1, \ldots, z_n}$ is associated to $\mathcal R$ which is supposed to 
be a rational form with poles at $z_i=z_j$, $i\ne j$ only. 
Thus the general principle is the following. By associating $z_i$ to $c_i(p)$ on a leaf or section, 
 and computing \eqref{torba}, we study its analytic behavior.  
If \eqref{torba} has poles then they could be related 
to singular points of $\mathcal M/\F$.  
Next, for \eqref{Inm}, \eqref{Jnm}, \eqref{pervayaforma}, and \eqref{vtorayaforma},
 for $z_i=c_i(p_i)$, we determine the domains of convergency. 
When such a domain is limited to one point, then $\mathcal M/\F$ might have a one point leaf. 
Finally, consider $\delta^0_1 \Phi$,
for $\Phi \in C^0_1(V, \W, \F)$, and identify $z$ to $c(w)$, where $c(w)$ is a local coordinate 
on a leaf or section. Then $\mathcal R(z)=\langle w', \delta^0_1 \Phi \rangle$ considered on the whole leaf 
may distinguish if it is compact or non-compact.  
Note that for the same $\Phi$ we may consider $c_i(p_i)$, $1 \le i \le n$ either on a leaf or section, 
i.e., in transversal directions on $\mathcal M/\F$. Thus, in case of saddle points, we have different values of, e.g.,  
integrals
\eqref{integral} in these directions. 
For cycles on a curve we determine 
if they are transversal to leaves of foliation by using the above considerations.
\section{Further directions}
\label{further}
There exist a few approches to definition and computation of cohomologies of vertex operator algebras. 
\cite{Huang, Li}.  
Taking into account the above definitions and construction, 
we aim to consideration of a characteristic classes theory for arbitrary codimension regular and singular 
 foliations vertex operator algebras. 
In this paper, we consider foliations of codimension one. Arbitrary codimension case will be considered elsewhere. 
Losik defines a smooth structure on the leaf space
$M/\mathcal{F}$ of a foliation $\mathfrak{F}$ of codimension $n$ on
a smooth manifold $\mathcal M$ that allows to apply to $M/\mathcal{F}$ the
same techniques as to smooth manifolds.
In \cite{Losik} characteristic classes for a
foliation as  elements of the cohomology of certain bundles over the
leaf space $M/\mathcal{F}$ are defined.  
It would be interesting also to develope intrinsic (i.e., purely coordinate independent) theory of a smooth manifold 
foliation cohomology involving vertex algebra bundles \cite{BZF}.
Similar to Losik's theory, we use bundles correlation functions) over a foliated space. 
The idea of studies of cohomology of certain bundles on a smooth manifold $\mathcal M$ and making connection to 
a cohomology of $\mathcal M$ has first appeared in \cite{BS}. 
 This can have a relation with Losik's work \cite{Losik} proposing a new framework for singular spaces and 
 characteristic classes. 
In applications, one would be interested in applying techniques of this paper to case of higher-dimensional 
manifolds of codimension one \cite{BG, BGG}. In particular, the question of higher non-vanishing invariants, as 
well as the problem of distinguishing of compact and non-compact leaves for the Reeb foliation of the full torus, 
are also of high importance. 
It would be important to establish connection to chiral de-Rham complex on a smooth manifold introduced in \cite{MSV}. 
After a modification, one is able to introduce a vertex algebra cohomology of smooth manifolds on a similar basis as in 
this paper.  

One can mention a possibility to derive differential equations \cite{Huang0} 
for characters on 
separate leaves of foliation. Such equations are derived for various genuses and can be used in frames 
of Vinogradov theory \cite{vinogradov}. 
The structure of foliation (in our sense) can be also studied from the automorphic function theory point of view. 
Since on separate leaves one proves automorphic properties of characters, on can think about "global" automorphic 
properties for the whole foliation. 
\section{Applications in Mathematical Physics}
\label{applications} 
\subsection{Applications in conformal field theory and moduli spaces}
The problem of classification of leaves of codimension one foliations 
finds its applications in mechanics \cite{N0, N1, N2}. 
The general considerations of this paper are also useful even apart from problems of 
classification of leaves of codimension one foliations.     
Reduction formulas directly related to connections in this paper 
are used in 
computations of correlation functions on compact Riemann surfaces in 
Conformal Field theory \cite{Knizhka, Zhu, BZF}.  
Starting from an $n$-point functions we are able to relate it to corresponding 
the partition function via the recurrsion procedure. 
Studies of the space of connections defined for correlation functions of a 
particular vertex algebra brings about invariants of foliations which 
represent invariants for corresponding Conformal Field Theory \cite{S1, S3}. 
This gives an important link between cohomology of foliations with 
vertex algebra cohomology 
 and the geometry of foliations in general. 
\subsection{Deformations of Lie algebras.}
The theory of connections on the space of correlation function characterising leaves of a foliation
find its relation to the deformations of Lie algebras. 
Such deformations are used in various applications in Mathematical Physics. 
 It is well known that the Lie algebra cohomology 
with values in the adjoint representation of a Lie algebra
 answers questions about deformations of such Lie algebra as an algebraic
object. 
There arise natural questions of this type for  
 bi-graded differential Lie algebras resulting from chain complex constructions.   
Bi-graded differential algebras results also from constructions of 
double complexes for continual Lie algebra that are used in exactly solvable models \cite{RS}. 
Using characterization of connection leaves, we will study rigidity of  
  bi-graded differential Lie algebras resulting from chain complex constructions. 
For a 
compact Riemann surface $\Sigma^{(g)}$
of genus $g\geq 2$, we expect to find a relations for cohomologies in terms of 
elements of Fr\'echet spaces 
 given by the polynomials on $T_{\Sigma^{(g)}}{\mathcal N}(g,0)$.
  It's the space of formal power series on $T_{\Sigma^{(g)}}{\mathcal N}(g,0)^*$. 
This could be interpreted as a relation between cohomology with
adjoint coefficients of a Lie algebra $\mathfrak{g}$, 
i.e., graded differential 
deformations of global sections of $\mathfrak{g}$, and deformations
of the underlying manifold.
The choice of the coefficients in the Lie algebra cohomology 
determines a geometric object on the moduli space in a formal
neighborhood of a point. 
 \subsection{Developments}
The recursion procedure \cite{Zhu} 
 plays a fundamental role in the theory of correlation functions for 
 vertex operator algebras. It provides relations between $n$- and $n-1$-point correlation functions. 
In our foliation picture, the recursion procedure brings about relations among leaves of foliation. 
We work in the  
approach of formulation and computation of cohomologies of vertex operator algebras.   
Taking into account the above definitions and construction, 
we would like to develop a theory of 
characteristic classes for vertex operator algebras.  
This can have a relation with Losik's work \cite{LosikArxiv} 
proposing a new framework for singular spaces and 
new kind of characteristic classes.  
Relations to \cite{BGG} on Reeb foliations modified Godbillon-Vey class and can be also considered. 
On the other hand, we would like to apply methodology of vertex algebras in order to 
complete Losik's theory of characteristic classes \cite{LosikArxiv}.  
One can mention a possibility to derive differential equations  
for correlation functions on 
separate leaves of foliation. Such equations are derived for various genera and can be used in frames 
of Vinogradov theory \cite{vinogradov}. 
If we consider 
 multipoint correlation function for vertex algebras on Riemann surfaces of arbitrary genus, 
and corresponding connections arrising from these correlation functions, then there 
exist also relations to 
Krichever-Novikov type algebras \cite{KN1, KN2, KN2, Schl0, Schl1, Schl2 }. These 
  are generalizations of the Witt, Virasoro, affine Lie algebras.  
In particular, one is able to use the structure of Krichever-Novikov type algebras (as higher-genus generalizations of 
algebras related to vertex algebras) to study foliated spaces introduced in this paper.
The construction of $n$-point functions on genus $g$ sewn Riemann surfaces 
can be used for introduction and computation of cohomology of Krichever-Novikov type algebras in appropriate setup. 
\section*{Acknowledgements}
The author would like to thank A. Galaev, H. V. L\^e, A. Lytchak, P. Somberg, 
and P. Zusmanovich for related discussions. 
Research of the author was supported by the GACR project 18-00496S and RVO: 67985840. 
\section{Appendix: Grading-restricted vertex algebras and their modules}
\label{grading}
In this section, following \cite{Huang} we recall basic properties of 
grading-restricted vertex algebras and their grading-restricted generalized 
modules, useful for our purposes in later sections. 
We work over the base field $\C$ of complex numbers. 

\begin{definition}
A vertex algebra   
$(V,Y_V,\mathbf{1}_V)$, (cf. \cite{K}),  consists of a $\Z$-graded complex vector space  
\[
V = \coprod_{n\in\Z}\,V_{(n)}, \quad \dim V_{(n)}<\infty, 
\]
 for each $n\in \Z$,   
and linear map 
\[
Y_V:V\rightarrow {\rm End \;}(V)[[z,z^{-1}]], 
\]
 for a formal parameter $z$ and a 
distinguished vector $\mathbf{1}_V\in V$.   
The evaluation of $Y_V$ on $v\in V$ is the vertex operator
\begin{equation}
\label{vop}
Y_V(v)\equiv Y_V(v,z) = \sum_{n\in\Z}v(n)z^{-n-1}, 
\end{equation}
with components $(Y_{V}(v))_{n}=v(n)\in {\rm End \;}(V)$, where $Y_V(v,z)\mathbf{1}_V = v+O(z)$.
\end{definition}
\begin{definition}
\label{grares}
A grading-restricted vertex algebra satisfies 
the following conditions:
\begin{enumerate}
\item {Grading-restriction condition}:
$V_{(n)}$ is finite dimensional for all $n\in \Z$, and $V_{(n)}=0$ for $n\ll 0$; 

\item { Lower-truncation condition}:
For $u$, $v\in V$, $Y_{V}(u, z)v$ contains only finitely many negative 
power terms, that is, 
\[
Y_{V}(u, z)v\in V((z)), 
\] 
(the space of formal 
Laurent series in $z$ with coefficients in $V$);   

\item { Identity property}: 
Let ${\rm Id}_{V}$ be the identity operator on $V$. Then
\[
Y_{V}(\mathbf{1}_V, z)={\rm Id}_{V};  
\]

\item { Creation property}: For $u\in V$, 
\[
Y_{V}(u, z)\mathbf{1}_V\in V[[z]],
\]  
and 
\[
\lim_{z\to 0}Y_{V}(u, z)\mathbf{1}_V= u;
\]

\item { Duality}: 
For $u_{1}, u_{2}, v\in V$, 
\[
v'\in V'=\coprod_{n\in \mathbb{Z}}V_{(n)}^{*}, 
\]
where 
 $V_{(n)}^{*}$ denotes
the dual vector space to $V_{(n)}$ and $\langle\, . ,  .\rangle$ the evaluation 
pairing $V'\otimes V\to \C$, the series 
\begin{eqnarray}
\label{porosyata}
& & \langle v', Y_{V}(u_{1}, z_{1})Y_{V}(u_{2}, z_{2})v\rangle,
\\
& & \langle v', Y_{V}(u_{2}, z_{2})Y_{V}(u_{1}, z_{1})v\rangle, 
\\
& & \langle v', Y_{V}(Y_{V}(u_{1}, z_{1}-z_{2})u_{2}, z_{2})v\rangle, 
\end{eqnarray}
are absolutely convergent
in the regions 
\[
|z_{1}|>|z_{2}|>0,
\]
\[ 
|z_{2}|>|z_{1}|>0,
\]
\[
|z_{2}|>|z_{1}-z_{2}|>0,
\]
 respectively, to a common rational function 
in $z_{1}$ and $z_{2}$ with the only possible poles at $z_{1}=0=z_{2}$ and 
$z_{1}=z_{2}$; 

\item { $L_V(0)$-bracket formula}: Let $L_{V}(0): V\to V$,  
be defined by 
\[
L_{V}(0)v=nv, \qquad n=\wt(v),  
\]
 for $v\in V_{(n)}$.  
Then
\[
[L_{V}(0), Y_{V}(v, z)]=Y_{V}(L_{V}(0)v, z)+z\frac{d}{dz}Y_{V}(v, z), 
\]
for $v\in V$. 

\item { $L_V(-1)$-derivative property}:  
Let 
\[
L_{V}(-1): V\to V, 
\]
 be the operator given by 
\[
L_{V}(-1)v=\res_{z}z^{-2}Y_{V}(v, z)\one_V=Y_{(-2)}(v) \one_V,  
\]
for $v\in V$. Then for $v\in V$, 
\begin{equation}
\label{derprop}
\frac{d}{dz}Y_{V}(u, z)=Y_{V}(L_{V}(-1)u, z)=[L_{V}(-1), Y_{V}(u, z)].
\end{equation}
\end{enumerate}
\end{definition}
In addition to that,  we recall here the following definition (cf. \cite{BZF}): 
\begin{definition}
 A grading-restricted vertex algebra $V$ is called conformal of central 
charge $c \in \C$,
 if there exists a non-zero conformal vector (Virasoro vector) $\omega \in V_{(2)}$ such that the
corresponding vertex operator 
\[
Y_V(\omega, z)=\sum_{n\in\Z}L_V(n)z^{-n-2}, 
\]
is determined by modes of Virasoro algebra $L_V(n): V\to V$ satisfying 
\[
[L_V(m), L_V(n)]=(m-n)L(m+n)+\frac{c}{12}(m^{3}-m)\delta_{m+b, 0}\; {\rm Id_V}. 
\]
\end{definition}
\begin{definition}
\label{primary}  
A vector $A$ which belongs to a module $W$ of a quasi-conformal 
 grading-restricted vertex algebra $V$ is called 
primary of conformal dimension $\Delta(A) \in  \mathbb Z_+$ if  
\begin{eqnarray*}
L_W(k) A &=& 0,\;  k > 0, 
\nn
 L_W(0) A &=& \Delta(A) A. 
\end{eqnarray*}
\end{definition}
\begin{definition}
A {grading-restricted generalized $V$-module} is a vector space 
$W$ equipped with a vertex operator map 
\begin{eqnarray*}
Y_{W}: V\otimes W&\to& W[[z, z^{-1}]],
\nn
u\otimes w&\mapsto & Y_{W}(u, w)\equiv Y_{W}(u, z)w=\sum_{n\in \Z}(Y_{W})_{n}(u,w)z^{-n-1}, 
\end{eqnarray*}
and linear operators $L_{W}(0)$ and $L_{W}(-1)$ on $W$ satisfying the following
conditions:
\begin{enumerate}
\item {Grading-restriction condition}:
The vector space $W$ is $\mathbb C$-graded, that is, 
\[
W=\coprod_{\alpha\in \mathbb{C}}W_{(\alpha)},
\]
 such that 
$W_{(\alpha)}=0$ when the real part of $\alpha$ is sufficiently negative; 

\item { Lower-truncation condition}:
For $u\in V$ and $w\in W$, $Y_{W}(u, z)w$ contains only finitely many negative 
power terms, that is, $Y_{W}(u, z)w\in W((z))$; 

\item { Identity property}: 
Let ${\rm Id}_{W}$ be the identity operator on $W$. 
Then 
\[
Y_{W}(\mathbf{1}_V, z)={\rm Id}_{W}; 
\] 

\item { Duality}: For $u_{1}, u_{2}\in V$, $w\in W$, 
\[
w'\in W'=\coprod_{n\in \mathbb{Z}}W_{(n)}^{*}, 
\]
 $W'$ denotes 
the dual $V$-module to $W$ and $\langle\, .,. \rangle$ their evaluation 
pairing, the series 
\begin{eqnarray*}
&& \langle w', Y_{W}(u_{1}, z_{1})Y_{W}(u_{2}, z_{2})w\rangle,
\\ 
&& \langle w', Y_{W}(u_{2}, z_{2})Y_{W}(u_{1}, z_{1})w\rangle, 
\\
&& \langle w', Y_{W}(Y_{V}(u_{1}, z_{1}-z_{2})u_{2}, z_{2})w\rangle, 
\end{eqnarray*}
are absolutely convergent
in the regions
\[ 
|z_{1}|>|z_{2}|>0,
\]
\[ 
|z_{2}|>|z_{1}|>0, 
\]
\[
|z_{2}|>|z_{1}-z_{2}|>0,
\]
 respectively, to a common rational function 
in $z_{1}$ and $z_{2}$ with the only possible poles at $z_{1}=0=z_{2}$ and 
$z_{1}=z_{2}$. 
\item { $L_{W}(0)$-bracket formula}: For  $v\in V$,
\[
[L_{W}(0), Y_{W}(v, z)]=Y_{W}(L_V(0)v, z)+z\frac{d}{dz}Y_{W}(v, z); 
\]
\item { $L_W(0)$-grading property}: For $w\in W_{(\alpha)}$, there exists
$N \in \Z_{+}$ such that 
\begin{equation}
\label{gradprop}
(L_{W}(0)-\alpha)^{N}w=0;  
\end{equation}
\item { 
$L_W(-1)$-derivative property}: For $v\in V$,
\begin{equation}
\label{derprop}
\frac{d}{dz}Y_{W}(u, z)=Y_{W}(L_{V}(-1)u, z)=[L_{W}(-1), Y_{W}(u, z)].
\end{equation}
\end{enumerate}
\end{definition}
The translation property of vertex operators 
\begin{equation}
\label{transl}
 Y_{W}(u, z) = e^{-z' L_{W}(-1)} Y_{W}(u, z+z') e^{z' L_{W}(-1)}, 
\end{equation}
for $z' \in \C$, follows from from \eqref{derprop}.  
For $v\in V$, and $w \in W$, the intertwining operator 
\begin{eqnarray}
\label{interop}
&& Y_{WV}^{W}: V\to W,  
\nn
&&
v   \mapsto  Y_{WV}^{W}(w, z) v,    
\end{eqnarray}
 is defined by 
\begin{eqnarray}
\label{wprop}
Y_{WV}^{W}(w, z) v= e^{zL_W(-1)} Y_{W}(v, -z) w. 
\end{eqnarray}
We will also use the following property of intertwining operators \eqref{interop} \cite{Hu3}. 
For a function $f(u)$, $u\in V$, 
\[
f\left(Y_V(u, z)\one_V \right) = Y_{WV}^{W}(f(u), z)\one_V.   
\]

Let us recall some further facts from \cite{BZF} relating generators of Virasoro algebra with the group of 
automorphisms in complex dimension one. 
 Let us represent an element of ${\rm Aut} \; \Oo^{(1)}$ by the map  
\begin{equation}
\label{lempa}
z \mapsto \rho=\rho(z),
\end{equation}
given by the power series
\begin{equation}
\label{prostoryad}
\rho(z) = \sum\limits_{k \ge 1} a_k z^k, 
\end{equation}
%
$\rho(z)$ can be represented in an exponential form 
\begin{equation}
\label{rog}
f(z) = \exp \left(  \sum\limits_{k > -1} \beta_{k }\; z^{k+1} \partial_{z} \right) 
\left(\beta_0 \right)^{z \partial_z}.z, 
\end{equation}
where we express $\beta_k \in \mathbb C$, $k \ge 0$, through combinations of $a_k$, $k\ge 1$.  
 A representation of Virasoro algebra modes in terms of differenatial operators is given by \cite{K} 
\begin{equation}
\label{repro}
L_W(m) \mapsto - \zeta^{m+1}\partial_\zeta, 
\end{equation}
for $m \in \Z$. 
 By expanding \eqref{rog} and comparing to \eqref{prostoryad} we obtain a system of equations which, 
 can be solved recursively for all $\beta_{k}$.
In \cite{BZF}, $v \in V$, they derive the formula 
\begin{eqnarray}
\label{infaction}
&&
\left[L_W(n), Y_W (v, z) \right] 
=  \sum_{m \geq -1}  
 \frac{1}{(m+1)!} \left(\partial^{m+1}_{z} z^{m+1}  \right)\;  
Y_W (L_V(m) v, z),    
\end{eqnarray}
of a Virasoro generator commutation with a vertex operator. 
Given a vector field 
\begin{equation}
\label{top}
\beta(z)\partial_z= \sum_{n \geq -1} \beta_n z^{n+1} \partial_z, 
\end{equation}
which belongs to local Lie algebra of ${\rm Aut}\; \Oo^{(1)}$, 
 one introduces the operator 
\[
\beta = - \sum_{n \geq -1} \beta_n L_W(n). 
\]
We conlclude from \eqref{top} with the following 
\begin{lemma}
\begin{eqnarray}
\label{infaction000}
&&
\left[\beta, Y_W (v, z) \right] 
=  \sum_{m \geq -1}  
 \frac{1}{(m+1)!} \left(\partial^{m+1}_{z} \beta(z)  \right)\;  
Y_W (L_V(m) v, z).   
\end{eqnarray}
\end{lemma}
The formula \eqref{infaction000}  is used in \cite{BZF} (Chapter 6) in order to prove invariance of 
vertex operators multiplied by conformal weight differentials in case of primary states, and 
in generic case.  

Let us give some further definition: 
\begin{definition}
\label{quasiconf}
A grading-restricted vertex algebra $V$-module $W$ is called quasi-conformal if 
  it carries an action of local Lie algebra of ${\rm Aut}\; \Oo$  
such that commutation formula
 \eqref{infaction000}    
 holds for any 
$v \in  V$, the element $L_W(-1) = - \partial_{z}$, 
as the translation operator $T$,
\[
L_W(0) = - z \partial_{z},  
\]
  acts semi-simply with integral
eigenvalues, and the Lie subalgebra of the positive part of local Lie algebra of ${\rm Aut}\; \Oo^{(n)}$
 acts locally nilpotently. 
\end{definition}
Recall \cite{BZF} the exponential form $f(\zeta)$ \eqref{rog} of the coordinate transformation \eqref{lempa} 
$\rho(z) \in {\rm Aut}\; \Oo^{(1)}$.
A quasi-conformal vertex algebra posseses the formula \eqref{infaction000}, thus 
it is possible 
by using the identification \eqref{repro}, to introduce the linear operator 
 representing $f(\zeta)$ \eqref{rog} on $\W_{z_1, \ldots, z_n}$,  
\begin{equation}
\label{poperator}
 P\left(f (\zeta) \right)= 
\exp \left( \sum\limits_{m >  0} (m+1)\; \beta_m \; L_V(m) \right) \beta_0^{L_W(0)},  
\end{equation}
 (note that we have a different normalization in it). 
In \cite{BZF} (Chapter 6) it was shown that  
the action of an operator similar to \eqref{poperator} 
on a vertex algebra element $v\in V_n$ contains finitely meny terms, and 
subspaces 
\[
V_{\le m} = \bigoplus_{ n \ge K}^m V_n, 
\]
 are stable under all operators $P(f)$, $f \in  {\rm Aut}\; \Oo^{(1)}$. 
 In \cite{BZF} they proved the following 
\begin{lemma}
 The assignment
\[
 f \mapsto P(f), 
\]
 defines a representation of ${\rm Aut}\; \Oo^{(1)}$
on $V$, 
\[
 P(f_1 * f_2) = P(f_1) \; P(f_2),  
\]
 which is the inductive limit of the
representations $V_{\le m}$, $m\ge K$.
\end{lemma}
Similarly, \eqref{poperator} provides a representation operator on $\W_{z_1, \ldots, z_n}$. 
 
\section{Appendix: $\W$-valued rational functions} 
\label{valued}
 Recall the definition of shuffles. 
Let  $S_{q}$ be the permutation group. 
For $l \in \N$ and $1\le s \le l-1$, let $J_{l; s}$ be the set of elements of 
$S_{l}$ which preserve the order of the first $s$ numbers and the order of the last 
$l-s$ numbers, that is,
\[
J_{l, s}=\{\sigma\in S_{l}\;|\;\sigma(1)<\cdots <\sigma(s),\;
\sigma(s+1)<\cdots <\sigma(l)\}.
\]
The elements of $J_{l; s}$ are called shuffles, and we use the notation 
\[
J_{l; s}^{-1}=\{\sigma\;|\; \sigma\in J_{l; s}\}.
\]

 We define the configuration spaces: 
\[
F_{n}\C=\{(z_{1}, \dots, z_{n})\in \C^{n}\;|\; z_{i}\ne z_{j}, i\ne j\},
\] 
for $n\in \Z_{+}$.
Let $V$ be a grading-restricted 
vertex algebra, and $W$ a a grading-restricted generalized $V$-module. 
By $\overline{W}$ we denote the algebraic completion of $W$, 
\[
\overline{W}=\prod_{n\in \mathbb C}W_{(n)}=(W')^{*}.
\]
\begin{definition}
 A $\overline{W}$-valued rational function in $(z_{1}, \dots, z_{n})$ 
with the only possible poles at 
$z_{i}=z_{j}$, $i\ne j$, 
is a map 
\begin{eqnarray*}
 f:F_{n}\C &\to& \overline{W},   
\\
 (z_{1}, \dots, z_{n}) &\mapsto& f(z_{1}, \dots, z_{n}),   
\end{eqnarray*} 
such that for any $w'\in W'$,  
\[
R(z_1, \ldots, z_n)= \langle w', f(z_{1}, \dots, z_{n}) \rangle,
\] 
is a rational function in $(z_{1}, \dots, z_{n})$  
with the only possible poles  at 
$z_{i}=z_{j}$, $i\ne j$. 
In this paper, such a map is called $\overline{W}$-valued rational function 
in $(z_{1}, \dots, z_{n})$ with possible other poles.
 The space of $\overline{W}$-valued rational functions is denoted by $\overline{W}_{z_{1}, \dots, z_{n}}$.
\end{definition} 
One defines an action of $S_{n}$ on the space $\hom(V^{\otimes n}, 
\overline{W}_{z_{1}, \dots, z_{n}})$ of linear maps from 
$V^{\otimes n}$ to $\overline{W}_{z_{1}, \dots, z_{n}}$ by 
\[
\sigma(\Phi)(v_{1}, z_1; \cdots; v_{n}, z_n)  
=\Phi (v_{\sigma(1)}, v_{\sigma(1)};  \cdots v_{\sigma(n)}, z_{\sigma(n)}), 
\]  
for $\sigma\in S_{n}$, and $v_{1}, \dots, v_{n}\in V$.
We will use the notation $\sigma_{i_{1}, \dots, i_{n}}\in S_{n}$, to denote the 
the permutation given by $\sigma_{i_{1}, \dots, i_{n}}(j)=i_{j}$,  
for $j=1, \dots, n$.
In \cite{Huang} one finds:
\begin{proposition} 
The subspace of $\hom(V^{\otimes n}, 
\overline{W}_{z_{1}, \dots, z_{n}})$ consisting of linear maps
having
the $L(-1)$-derivative property, having the $L(0)$-conjugation property
or being composable with $m$ vertex operators is invariant under the 
action of $S_{n}$.
\end{proposition}


Let us introduce another definition: 
\begin{definition}
\label{wspace}
We define the space $\W_{z_1, \dots, z_n}$ of 
  $\overline{W}_{z_{1}, \dots, z_{n}}$-valued rational forms $\Phi$ 
with each vertex algebra element entry $v_i$, $1 \le i \le n$
of a quasi-conformal grading-restricted vertex algebra $V$ tensored with power $\wt(v_i)$-differential of 
corresponding formal parameter $z_i$, i.e., 
\begin{equation}
\label{bomba}
\Phi \left(dz_1^{\wt(v_1)} \otimes v_{1}, z_1; \cdots;
 dz_n^{\wt(v_n)} \otimes  v_{n}, z_n\right) \in \W_{z_1, \dots, z_n}. 
\end{equation}
We assume also that \eqref{bomba} 
satisfy $L_V(-1)$-derivative \eqref{lder1}, $L_V(0)$-conjugation 
\eqref{loconj} properties, and the symmetry property 
with respect to action of the symmetric group $S_n$: 
\begin{equation}
\label{shushu} 
\sum_{\sigma\in J_{l; s}^{-1}}(-1)^{|\sigma|}
 \left(\Phi(v_{\sigma(1)}, z_{\sigma(1)}; \cdots; v_{\sigma(l)},  z_{\sigma(1)}) \right)=0.
\end{equation}
\end{definition}
In Section \ref{spaces} we prove that \eqref{bomba} is invariant with respect to changes of formal parameters 
$(z_1, \dots, z_n)$. 
\section{Appendix: Properties of matrix elements for a grading-restricted vertex algebra}
\label{properties}
Let $V$ be a grading-restricted vertex algebra and $W$ a grading-restricted generalized $V$-module. 
Let us recall some definitions and facts about matrix elements for a grading-restricted vertex algebra \cite{Huang}. 
If a meromorphic function $f(z_{1}, \dots, z_{n})$ on a domain in $\mathbb C^{n}$ is 
 analytically extendable to a rational function in $z_{1}, \dots, z_{n}$, 
we denote this rational function  
by $R(f(z_{1}, \dots, z_{n}))$.
Let us recall a few definitions from \cite{Huang}
\begin{definition}
\label{}
For $n\in \Z_{+}$,  
a linear map 
\[
\Phi(v_{1}, z_{1};  \ldots ; v_{n},  z_{n})
= V^{\otimes n}\to 
\W_{z_{1}, \dots, z_{n}}, 
\]
 is said to have
the  $L(-1)$-derivative property if
\begin{equation}
\label{lder1}
(i) \qquad \langle w', \partial_{z_{i}} \Phi(v_{1}, z_{1};  \ldots ; v_{n},  z_{n})\rangle= 
\langle w',  \Phi(v_{1}, z_{1};  \ldots; L_{V}(-1)v_{i}, z_i; \ldots ; v_{n},  z_{n}) \rangle, 
\end{equation}
for $i=1, \dots, n$, $v_{1}, \dots, v_{n}\in V$, $w'\in W$, 
and  
\begin{eqnarray}
\label{lder2}
(ii) \qquad \sum\limits_{i=1}^n\partial_{z_{i}} \langle w', \Phi(v_{1}, z_{1};  \ldots ; v_{n},  z_{n})\rangle=
\langle w', L_{W}(-1).\Psi(v_{1}, z_{1};  \ldots ; v_{n},  z_{n}) \rangle, 
\end{eqnarray}
with some action $.$ of $L_{W}(-1)$  on $\Phi(v_{1}, z_{1};  \ldots ; v_{n},  z_{n})$, and 
and  $v_{1}, \dots, v_{n}\in V$. 
\end{definition}
Note that since $L_{W}(-1)$ is a weight-one operator on $W$, 
for any $z\in \C$,  $e^{zL_{W}(-1)}$ is a well-defined linear operator
on $\overline{W}$. 

In \cite{Huang} we find the following 
\begin{proposition}
Let $\Phi$ be a linear map having the $L(-1)$-derivative property. Then for $v_{1}, \dots, v_{n}\in V$,
$w'\in W'$, $(z_{1}, \dots, z_{n})\in F_{n}\C$, $z\in \C$ such that
$(z_{1}+z, \dots,  z_{n}+z)\in F_{n}\C$,
\begin{eqnarray}
\label{ldirdir}
 \langle w', e^{zL_{W}(-1)} 
\Phi \left(v_{1}, z_1; \ldots; v_{n}, z_{n} \right) \rangle 
%
=\langle w', 
\Phi(v_{1}, z_{1}+z ; \ldots; v_{n},  z_{n}+z) \rangle,  
\end{eqnarray}
and for $v_{1}, \dots, v_{n}\in V$,
$w'\in W'$, $(z_{1}, \dots, z_{n})\in F_{n}\C$, $z\in \C$,  and $1\le i\le n$ such that
\[
(z_{1}, \dots, z_{i-1}, z_{i}+z, z_{i+1}, \dots, z_{n})\in F_{n}\C,
\]
the power series expansion of 
\begin{equation}\label{expansion-fn}
\langle w', 
 \Phi(v_{1}, z_1;  \ldots; v_{i-1}, z_{i-1};  v_i, z_{i}+z; v_{v+1}, z_{i+1}; \ldots  v_{n}, z_n) \rangle, 
\end{equation}
in $z$ is equal to the power series
\begin{equation}\label{power-series}
\langle w', 
\Phi(v_{1} z_1; \ldots;  v_{i-1}, z_{i-1}; e^{zL(-1)}v_{i}, z_i;  
 v_{i+1}, z_{i+1};  \ldots; v_{n}, z_n) \rangle,  
\end{equation}
in $z$.
In particular, the power series (\ref{power-series}) in $z$ is absolutely convergent
to (\ref{expansion-fn}) in the disk $|z|<\min_{i\ne j}\{|z_{i}-z_{j}|\}$. 
\end{proposition}

Finally, we have 
\begin{definition}
A linear map 
\[
\Phi: V^{\otimes n} \to \W_{z_{1}, \dots, z_{n}}
\]
 has the  $L(0)$-conjugation property if for $v_{1}, \dots, v_{n}\in V$,
$w'\in W'$, $(z_{1}, \dots, z_{n})\in F_{n}\C$ and $z\in \C^{\times}$ so that 
$(zz_{1}, \dots, zz_{n})\in F_{n}\C$,
\begin{eqnarray}
\label{loconj}
\langle w', z^{L_{W}(0)}   
\Phi \left(v_{1}, z_1; \cdots; v_{n}, z_{n} \right) \rangle 
=\langle w', 
 \Phi(z^{L(0)} v_{1}, zz_{1};  \cdots ;  z^{L(0)} v_{n},  zz_{n})\rangle.
\end{eqnarray} 
\end{definition}
\subsection{$E$-elements}
 For $w\in W$, the $\overline{W}$-valued function 
is given by 
$$
E^{(n)}_{W}(v_{1}, z_1; \cdots; v_{n}, z_n; w)
= E(\omega_{W}(v_{1}, z_{1}) \ldots \omega_{W}(v_{n}, z_{n})w),   
$$
where an element $E(.)\in \overline{W}$ is given by (see notations for $\omega_W$ in Section \ref{spaces}) 
\[
\langle w', E(.) \rangle =R(\langle w', . \rangle), 
\]
 and $R(.)$ denotes the following (cf. \cite{Huang}).   
Namely, 
if a meromorphic function $f(z_{1}, \dots, z_{n})$ on a region in $\C^{n}$
can be analytically extended to a rational function in $(z_{1}, \dots, z_{n})$, 
then the notation $R(f(z_{1}, \dots, z_{n}))$ is used to denote such rational function. 
One defines  
\[
E^{W; (n)}_{WV}(w; v_{1}, z_1 ; \ldots; v_{n}, z_n) 
=E^{(n)}_{W}(v_{1}, z_1; \ldots;  v_{n}, z_n; w),
\]
where 
$E^{W; (n)}_{WV}(w; v_{1}, z_1 ; \ldots; v_{n}, z_n)$ is  
an element of $\overline{W}_{z_{1}, \dots, z_{n}}$.
One defines
\[
 \Phi\circ \left(E^{(l_{1})}_{V;\;\one}\otimes \cdots \otimes E^{(l_{n})}_{V;\;\one}\right): 
V^{\otimes m+n}\to 
\overline{W}_{z_{1},  \dots, z_{m+n}},
\] 
by
\begin{eqnarray*}
\lefteqn{(\Phi\circ (E^{(l_{1})}_{V;\;\one}\otimes \cdots \otimes 
E^{(l_{n})}_{V;\;\one}))(v_{1}\otimes \cdots \otimes v_{m+n-1})}
\nn
&&=E(\Phi(E^{(l_{1})}_{V; \one}(v_{1}\otimes \cdots \otimes v_{l_{1}})\otimes \cdots
\nn 
&&\quad\quad\quad\quad\quad \otimes 
E^{(l_{n})}_{V; \one}
(v_{l_{1}+\cdots +l_{n-1}+1}\otimes \cdots 
\otimes v_{l_{1}+\cdots +l_{n-1}+l_{n}}))),  
\end{eqnarray*}
and 
\[
E^{(m)}_{W} \circ_{0 
} \Phi: V^{\otimes m+n}\to 
\overline{W}_{z_{1}, \dots, 
z_{m+n-1}},
\]
 is given by 
\begin{eqnarray*}
\lefteqn{
(E^{(m)}_{W}\circ_{0 
}\Phi)(v_{1}\otimes \cdots \otimes v_{m+n})
}\nn
&&
=E(E^{(m)}_{W}(v_{1}\otimes \cdots\otimes v_{m};
\Phi(v_{m+1}\otimes \cdots\otimes v_{m+n}))). 
\end{eqnarray*}
Finally,  
\[
E^{W; (m)}_{WV}\circ_{m+1 
}\Phi: V^{\otimes m+n}\to 
\overline{W}_{z_{1}, \dots, 
z_{m+n-1}},
\]
 is defined by 
\begin{eqnarray*}
(E^{W; (m)}_{WV}\circ_{m+1 
}\Phi)(v_{1}\otimes \cdots \otimes v_{m+n})
 =E(E^{W; (m)}_{WV}(\Phi(v_{1}\otimes \cdots\otimes v_{n})
; v_{n+1}\otimes \cdots\otimes v_{n+m})). 
\end{eqnarray*}
In the case that $l_{1}=\cdots=l_{i-1}=l_{i+1}=1$ and $l_{i}=m-n-1$, for some $1 \le i \le n$,
we will use $\Phi\circ_{i} E^{(l_{i})}_{V;\;\one}$ to 
denote $\Phi\circ (E^{(l_{1})}_{V;\;\one}\otimes \cdots 
\otimes E^{(l_{n})}_{V;\;\one})$.
Note that our notations differ with that of \cite{Huang}. 
\section{Appendix: Maps composable with vertex operators}
\label{composable}
%

In the construction of double complexes in Section \ref{coboundary}
 we would like to use linear maps from tensor powers of $V$ to the space 
$\W_{z_{1}, \dots, z_{n}}$ 
to define cochains in vertex algebra cohomology theory.
 For that purpose, in particular, to define the coboundary operator, we have to compose cochains 
with vertex operators. However, as mentioned in \cite{Huang},
 the images of vertex operator maps in general do not belong to 
algebras or thier modules.
 They belong to corresponding algebraic completions which constitute 
  one of the most subtle features of the theory of vertex algebras.
 Because of this, we might not be able to compose vertex operators directly. 
In order to overcome this problem \cite{H2}, we first write a series by projecting
an element of the algebraic completion of an algebra or a module 
to its homogeneous components. 
Then we compose these homogeneous components with vertex operators,  
and take formal sums.
 If such formal sums are absolutely convergent, then these operators 
can be composed and can be used in constructions.

Another question that appears is the question of associativity. 
Compositions of maps are usually associative.
 But for compositions of maps defined by sums of 
absolutely convergent series the existence of does not provide associativity in general.  
 Nevertheless, the requirement of analyticity provides the associativity \cite{Huang}.  
\begin{definition}
\label{composabilitydef}
For a $V$-module 
\[
W=\coprod_{n\in \C} W_{(n)}, 
\]
 and $m\in \C$, 
let
\[
P_{m}: \overline{W}\to W_{(m)}, 
\]
 be the projection from 
$\overline{W}$ to $W_{(m)}$. 
Let 
\[
\Phi: V^{\otimes n} \to \W_{z_{1}, \dots, z_{n}}, 
\]
 be a linear map. For $m\in \N$, 
$\Phi$ is called \cite{Huang, FQ} to be composable with $m$ vertex operators if  
the following conditions are satisfied:

\medskip 
1) Let $l_{1}, \dots, l_{n}\in \Z_+$ such that $l_{1}+\cdots +l_{n}=m+n$,
$v_{1}, \dots, v_{m+n}\in V$ and $w'\in W'$. Set 
 \begin{eqnarray}
\label{psii}
\Psi_{i}
&
=
&
E^{(l_{i})}_{V}(v_{k_1}, z_{k_1}- \zeta_{i};  
 \ldots;  
v_{k_i}, z_{k_i}- \zeta_{i} 
 ; \one_{V}),    
\end{eqnarray}
where
\begin{eqnarray}
\label{ki}
 {k_1}={l_{1}+\cdots +l_{i-1}+1}, \quad  \ldots, \quad  {k_i}={l_{1}+\cdots +l_{i-1}+l_{i}}, 
\end{eqnarray} 
for $i=1, \dots, n$. Then there exist positive integers $N^n_m(v_{i}, v_{j})$
depending only on $v_{i}$ and $v_{j}$ for $i, j=1, \dots, k$, $i\ne j$ such that the series 
\begin{eqnarray}
\label{Inm}
\mathcal I^n_m(\Phi)=
\sum_{r_{1}, \dots, r_{n}\in \Z}\langle w', 
\Phi(P_{r_{1}}\Psi_{1}; \zeta_1; 
 \ldots; 
P_{r_n} \Psi_{n}, \zeta_{n}) 
\rangle,
\end{eqnarray} 
is absolutely convergent  when 
\begin{eqnarray}
\label{granizy1}
|z_{l_{1}+\cdots +l_{i-1}+p}-\zeta_{i}| 
+ |z_{l_{1}+\cdots +l_{j-1}+q}-\zeta_{i}|< |\zeta_{i}
-\zeta_{j}|,
\end{eqnarray} 
for $i,j=1, \dots, k$, $i\ne j$ and for $p=1, 
\dots,  l_i$ and $q=1, \dots, l_j$. 
The sum must be analytically extended to a
rational function
in $(z_{1}, \dots, z_{m+n})$,
 independent of $(\zeta_{1}, \dots, \zeta_{n})$, 
with the only possible poles at 
$z_{i}=z_{j}$, of order less than or equal to 
$N^n_m(v_{i}, v_{j})$, for $i,j=1, \dots, k$,  $i\ne j$.

\medskip 
2) 
 For $v_{1}, \dots, v_{m+n}\in V$, there exist 
positive integers $N^n_m(v_{i}, v_{j})$, depending only on $v_{i}$ and 
$v_{j}$, for $i, j=1, \dots, k$, $i\ne j$, such that for $w'\in W'$, and 
\[
{\bf v}_{n,m}=(v_{1+m}\otimes \cdots\otimes v_{n+m}),
\]
\[
  {\bf z}_{n,m}=(z_{1+m}, \dots, z_{n+m}),
\]
 such that 
\begin{eqnarray}
\label{Jnm}
\mathcal J^n_m(\Phi)=  
\sum_{q\in \C}\langle w', E^{(m)}_{W} \Big(v_{1}\otimes \cdots\otimes v_{m}; 
P_{q}( \Phi({\bf v}_{n,m} ) ({\bf z}_{n,m})\Big)\rangle, 
\end{eqnarray}
is absolutely convergent when 
\begin{eqnarray}
\label{granizy2}
z_{i}\ne z_{j}, \quad i\ne j, \quad 
\nn
|z_{i}|>|z_{k}|>0, 
\end{eqnarray}
 for $i=1, \dots, m$, and $k=m+1, \dots, m+n$, and the sum can be analytically extended to a
rational function 
in $(z_{1}, \dots, z_{m+n})$ with the only possible poles at 
$z_{i}=z_{j}$, of orders less than or equal to 
$N^n_m(v_{i}, v_{j})$, for $i, j=1, \dots, k$, $i\ne j$,. 
\end{definition}
In \cite{Huang}, we the following useful proposition is proven: 
\begin{proposition}\label{comp-assoc}
Let $\Phi: V^{\otimes n}\to 
\overline{W}_{z_{1}, \dots, z_{n}}$
be composable with $m$ vertex operators. Then we have:
\begin{enumerate}
\item For $p\le m$, $\Phi$ is 
composable with $p$ vertex operators and for 
$p, q\in \Z_{+}$ such that $p+q\le m$ and 
$l_{1}, \dots, l_{n} \in \Z_+$ such that $l_{1}+\cdots +l_{n}=p+n$,
$\Phi\circ (E^{(l_{1})}_{V;\;\one}\otimes 
\cdots \otimes E^{(l_{n})}_{V;\;\one})$ and $E^{(p)}_{W}\circ_{p+1}\Phi$
are
composable with $q$ vertex operators. 

\item For $p, q\in \Z_{+}$ such that $p+q\le m$,
$l_{1}, \dots, l_{n} \in \Z_+$ such that $l_{1}+\cdots +l_{n}=p+n$ and
$k_{1}, \dots, k_{p+n} \in \Z_+$ such that $k_{1}+\cdots +k_{p+n}=q+p+n$,
we have
\begin{eqnarray*}
&(\Phi\circ (E^{(l_{1})}_{V;\;\one}\otimes 
\cdots \otimes E^{(l_{n})}_{V;\;\one}))\circ 
(E^{(k_{1})}_{V;\;\one}\otimes 
\cdots \otimes E^{(k_{p+n})}_{V;\;\one})&\nn
&=\Phi\circ (E^{(k_{1}+\cdots +k_{l_{1}})}_{V;\;\one}\otimes 
\cdots \otimes E^{(k_{l_{1}+\cdots +l_{n-1}+1}+\cdots +k_{p+n})}_{V;\;\one}).&
\end{eqnarray*}

\item For $p, q\in \Z_{+}$ such that $p+q\le m$ and
$l_{1}, \dots, l_{n} \in \Z_+$ such that $l_{1}+\cdots +l_{n}=p+n$,
we have
$$E^{(q)}_{W}\circ_{q+1} (\Phi\circ (E^{(l_{1})}_{V;\;\one}\otimes 
\cdots \otimes E^{(l_{n})}_{V;\;\one}))
=(E^{(q)}_{W}\circ_{q+1} \Phi)\circ (E^{(l_{1})}_{V;\;\one}\otimes 
\cdots \otimes E^{(l_{n})}_{V;\;\one}).$$

\item For $p, q\in \Z_{+}$ such that $p+q\le m$, we have
$$E^{(p)}_{W}\circ_{p+1} (E^{(q)}_{W}\circ_{q+1}\Phi)
=E^{(p+q)}_{W}\circ_{p+q+1}\Phi.$$
\end{enumerate}
\end{proposition}
\section{Appendix: Proofs of Lemmas \ref{empty}, \ref{nezu}, \ref{subset} and Proposiiton \ref{nezc}}
\label{proof}
In this Appendix we provide proofs of Lemma \ref{nezu} and Proposiiton \ref{nezc}
We start with the proof of Lemma \ref{empty}. 
\begin{proof} 
From the construction 
 of spaces for double complex for a grading-restricted vertex algebra cohomology, 
it is clear that the spaces $C^n (V, \W, \U, \F)(U_j)$,  $1 \le s \le m$ in Definition \ref{initialspace} are non-zero.
On each transversal section $U_s$,  $1 \le s \le m$, $\Phi(v_1, c_j(p_1);  \ldots; v_n, c_j(p_n))$
 belongs to the space 
$\W_{c_j(p_1),  \ldots, c_j(p_n)}$, and satisfy the $L(-1)$-derivative \eqref{lder1} and $L(0)$-conjugation
 \eqref{loconj}
properties. 
 A map $\Phi(v_1, c_j(p_1)$;  $\ldots$; $v_n, c_j(p_n))$ 
is composable with $m$ vertex operators 
with formal parameters identified with local coordinates $c_j(p'_j)$, on each transversal section $U_j$.
Note that on each transversal section, $n$ and $m$ the spaces \eqref{ourbicomplex} remain the same.   
The only difference may be constituted by the composibility conditions \eqref{Inm} and \eqref{Jnm} for  $\Phi$.

In particular, 
for  $l_{1}, \dots, l_{n}\in \Z_+$ such that $l_{1}+\cdots +l_{n}=n+m$, 
$v_{1}, \dots, v_{m+n}\in V$ and $w'\in W'$, recall \eqref{psii} that 
 \begin{eqnarray}
\label{psiii}
\Psi_{i}
&
=
&
\omega_V(v_{k_1}, c_{k_1}(p_{k_1})-\zeta_i)  \ldots  \omega_V(v_{k_i}, c_{k_i}(p_{k_i})-\zeta_i) \;\one_{V},     
\end{eqnarray}
where $k_i$ is defined in \eqref{ki}, 
for $i=1, \dots, n$, 
 depend on coordinates of points on transversal sections. 
At the same time, in the first composibility condition 
 \eqref{Inm} depends on projections $P_r(\Psi_i)$, $r\in \C$, 
 of $\W_{c(p_1), \ldots, c(p_n)}$ to $W$, and 
on arbitrary variables $\zeta_i$, $1 \le i \le m$.  
On each transversal connection $U_s$, $1 \le s \le m$,  
the absolute convergency is assumed for the series  \eqref{Inm} (cf. Appendix \ref{composable}). 
Positive integers $N^n_m(v_{i}, v_{j})$,
(depending only on $v_{i}$ and $v_{j}$) as well as $\zeta_i$, 
 for $i$, $j=1, \dots, k$, $i\ne j$, 
may vary for transversal sections $U_s$.  
Nevertheless, the domains of convergency determined by the conditions \eqref{granizy1} which have the form 
\begin{equation}
\label{popas}
|c_{m_i}(p_{m_i})-\zeta_{i}| 
+ |c_{n_i}(p_{n_i})-\zeta_{i}|< |\zeta_{i}-\zeta_{j}|,
\end{equation} 
for $m_i= l_{1}+\cdots +l_{i-1}+p$, $n=l_{1}+\cdots +l_{j-1}+q$,   
 $i$, $j=1, \dots, k$, $i\ne j$ and for $p=1, 
\dots,  l_i$ and $q=1, \dots, l_j$, 
are limited by $|\zeta_{i}-\zeta_{j}|$ in \eqref{popas} from above. 
Thus, for the intersection variation of sets of homology embeddings in \eqref{ourbicomplex}, 
 the absolute convergency condition for \eqref{Inm} is still fulfilled. 
Under intersection in \eqref{ourbicomplex}
by choosing appropriate $N^n_m(v_{i}, v_{j})$, 
one can analytically extend \eqref{Inm}  
to a rational function in $(c_1(p_{1}), \dots, c_{n+m}(p_{n+m}))$,
 independent of $(\zeta_{1}, \dots, \zeta_{n})$, 
with the only possible poles at 
$c_{i}(p_i)=c_{j}(p_j)$, of order less than or equal to 
$N^n_m(v_{i}, v_{j})$, for $i,j=1, \dots, k$,  $i\ne j$. 

As for the second condition in Definition of composibility, we note that, on each transversal section, 
the domains of absolute convergensy 
$c_{i}(p_i)\ne c_{j}(p_j)$, $i\ne j$
$|c_{i}(p_i)| > |c_{k}(p_j)|>0$, for 
 $i=1, \dots, m$, 
and 
$k=1+m, \dots, n+m$, 
for 
\begin{eqnarray}
\mathcal J^n_m(\Phi) &=&  
\sum_{q\in \C}\langle w',
\omega_W(v_{1}, c_1(p_1)) \ldots  \omega_W(v_{m}, c_m(p_m)) \;    
\nn
&& \qquad 
P_{q}(\Phi( v_{1+m}, c_{1+m}(p_{1+m}); \ldots; v_{n+m}, c_{n+m}(p_{n+m})) \rangle, 
\end{eqnarray}
are limited from below by the same set ot absolute values of local coordinates on 
transversal section. 
Thus, under intersection in \eqref{ourbicomplex} this condition is preserved, and 
  the sum \eqref{Jnm} can be analytically extended to a
rational function 
in $(c_{1}(p_1)$, $\dots$, $c_{m+n}(p_{m+n}))$ with the only possible poles at
$c_{i}(p_i)=c_{j}(p_j)$, 
of orders less than or equal to 
$N^n_m(v_{i}, v_{j})$, for $i, j=1, \dots, k$, $i\ne j$. 
Thus, we proved the lemma.   
\end{proof}
Next we give proof of Lemma \ref{nezu}. 
\begin{proof}
Suppose we consider another transversal basis $\U'$ for $\F$.  
According to the definition,   
for each transversal section $U_i$ which belong to the original basis $\U$ in \eqref{ourbicomplex}   
 there exists a holonomy embedding 
\[
{h}'_i: U_i \hookrightarrow U_j', 
\]
i.e., it embeds $U_i$ into a section $U_j'$ of our new transversal basis $\U'$.  
Then consider the sequnce of holonomy embeddings $\left\{ h'_k \right\}$ such that 
\[
U'_{0} \stackrel{h'_1}{\hookrightarrow}   \ldots \stackrel{h'_k}{\hookrightarrow} U'_k.   
\]
For the combination of embeddings $\left\{ 
{h}'_i, i \ge 0 \right\}$ and  
\[
U_{0}\stackrel{h_1}{\hookrightarrow}  \ldots \stackrel{h_k}{\hookrightarrow} U_k,    
\]
we obtain commutative diagrams. 
Since the intersection in \eqref{ourbicomplex} is performed over all sets of homology mappings, 
then it is independent on the choice of a  transversal basis.
\end{proof}
Next, we prove Proposition \ref{nezc}. 
\begin{proof}
Here we prove that for generic elements of a quasi-conformal grading-restricted vertex algebra 
 $\Phi$ and $\omega_W$ $\in$ $\W_{z_1, \ldots, z_n}$ and  
   are canonical, i.e., independent on 
 changes 
\begin{eqnarray}
\label{zwrho}
z_i \mapsto w_i=\rho(z_i), \quad 1 \le i \le n, 
\end{eqnarray}
 of local coordinates of  $c_i(p_i)$ and $c_j(p'_j)$ at points $p_i$ and $p'_j$, $1 \le i \le n$, 
 $1 \le j \le k$. 
Thus the construction of the double complex spaces \eqref{ourbicomplex} is proved to be canonical too.
Let us denote by 
\[
\xi_i = \left( \beta_0^{-1} \; dw_i \right)^{\wt(v_i)}. 
\]
Recall the linear operator \eqref{poper} (cf. Appendix \ref{grading}).   
Define introduce the action of the transformations \eqref{zwrho} as 
\begin{eqnarray}
\label{deiv}
&& \Phi \left(dw_1^{\wt(v_1)} \otimes  v_1, w_1;
 \ldots ; dw_n^{\wt(v_n)} \otimes v_n, w_n \right)  
\nn
&& \qquad =\left( \frac{d f(\zeta)}{d\zeta} \right)^{-L_W(0)}\; P(f(\zeta))  
  \; \Phi \left( \xi_1 
\otimes  v_1, z_1;  
\ldots ; \xi_n 
\otimes v_n, z_n \right). 
\end{eqnarray}
We then obtain 
\begin{lemma}
An element \eqref{bomba} 
\[
\Phi \left(dz_1^{\wt(v_1)} \otimes  v_1, z_1;
 \ldots ; dz_n^{\wt(v_n)} \otimes v_n, z_n \right), 
\]
of $\W_{z_1, \ldots, z_n}$ is canonical 
is invariant under transformations \eqref{zwrho} of $\left({\rm Aut}\; \Oo^{(1)}\right)^{\times n}$. 
\end{lemma}
\begin{proof}
Consider \eqref{deiv}. 
First, note that 
\[
 f'(\zeta) =\frac{df(\zeta)}{d\zeta}= \sum\limits_{m \ge 0} (m+1) \; \beta_m \zeta^m. 
\]
By using the identification \eqref{repro} and 
 and the $L_W(-1)$-properties \eqref{lder1} and \eqref{loconj} we  
obtain 
\begin{eqnarray*}
&& 
 \langle w',
 \Phi \left( dw_1^{\wt(v_1)} \otimes  v_1, w_1; 
 \ldots; dw_n^{\wt(v_n)} \otimes v_n, w_n \right) 
\rangle 
\end{eqnarray*}
\begin{eqnarray*}
&& = 
  \langle w', f'(\zeta)
^{-L_W(0)}\;  
P(f(\zeta)) 
  \; \Phi \left( \xi_1
\otimes  v_1, z_1;
 \ldots ; \xi_n 
 \otimes v_n, z_n \right)  
\rangle  
\end{eqnarray*}
\begin{eqnarray*}
&&  \qquad = 
 \langle w', 
\left(
\frac{d f(\zeta)} {d\zeta} \right)^{-L_W(0)} \; 
\Phi \left(dw_1^{\wt(v_1)} \otimes  v_1, 
 \sum\limits_{m \ge 0} (m+1)\; \beta_m z_1^{m+1}; 
 \ldots; \right. 
\nn
&& 
\left.  \left. 
\qquad \qquad \qquad \qquad \qquad \qquad \qquad \qquad  dw_n^{\wt(v_n)} \otimes v_n, 
 \sum\limits_{m \ge 0} (m+1)\; \beta_m z_n^{m+1} \right)  \rangle \right) 
\end{eqnarray*}
\begin{eqnarray*}
&& 
= 
 \langle w',  
\left( \frac{d f(\zeta)}{d \zeta} \right)^{-L_W(0)} \; 
\Phi \left(dw_1^{\wt(v_1)} \otimes  v_1, 
\left( \frac{d f(z_1)}{dz_1} \right) z_1;      \right. 
\nn
&& 
\left.  
\qquad \qquad \qquad \qquad \qquad \qquad \qquad \qquad  
 \ldots;  dw_n^{\wt(v_n)} \otimes v_n, 
 \left( \frac{d f(z_n)}{dz_n} \right) z_n\right)  \rangle 
\end{eqnarray*}
\begin{eqnarray*}
&&
\qquad =  
 \langle w',  
\Phi \left( \left( \frac{d f(z_1)}{dz_i}  \; dw_1 \right)^{-\wt(v_1)} \otimes  v_1, z_1;     
\right. 
\nn
&&
\left.  
\qquad \qquad \qquad \qquad \qquad \qquad \qquad \qquad \qquad
 \ldots; \left( \frac{d f(z_n)}{dz_n} \; dw_n \right)^{-\wt(v_n)}   \otimes v_n, z_n 
 \rangle \right) 
\nn
&& 
\qquad = 
  \langle w',  \Phi 
\left(dz_1^{\wt(v_1)} \otimes  v_1, z_1;
 \ldots; dz_n^{\wt(v_n)} \otimes v_n, z_n \right) \rangle. 
\end{eqnarray*}
Thus we proved the Lemma.  
\end{proof}
The elements $\Phi(v_1, z_1; \ldots; v_n, z_n)$ of $C^n_k(V, \W, \F)$ belong to the space $\W_{z_1, \ldots, z_n}$ 
and assumed to be composable with a set of vertex operators $\omega_W(v_j', c_j(p_j'))$, $1 \le j \le k$. 
Vertex operators $\omega_W(dc(p)^{\wt(v')} \otimes v_j', c_j(p_j'))$ constitute particular examples \
of mapping of $C^1_\infty(V, \W, \F)$ and, therefore, are invariant with respect to \eqref{zwrho}.  
Thus, the construction of spaces \eqref{ourbicomplex} is invariant under the action of the group 
\end{proof}
Finally, we give a proof of Lemma \ref{subset}. 
\begin{proof}
Since $n$ is the same for both spaces in \eqref{susus}, it only remains 
to check that the conditions for \eqref{Inm} and \eqref{Jnm} for $\Phi(v_1, c_j(p_1);$  $\ldots;$ $v_n, c_j(p_n))$
of  composibility Definition \ref{composable} 
 with vertex operators are stronger for ${C}_{m}^{n}(V, \W, \U, \F)$ 
then for ${C}_{m-1}^{n}(V, \W, \U, \F)$.  
In particular, 
in the first condition for \eqref{Inm} in definition of composability \ref{composabilitydef}
the difference between the spaces in \eqref{susus} is in indeces.  
Consider 
\eqref{psiii}. 
For ${C}_{m-1}^{n}(V, \W, \U, \F)$, the summations in idexes  
${k_1}={l_{1}+\cdots +l_{i-1}+1}$, ..., ${k_i}={l_{1}+\cdots +l_{i-1}+l_{i}}$, 
for the coordinates $c_j(p_1)$, ..., $c_j(p_n)$ 
with  $l_{1}, \dots, l_{n}\in \Z_+$, such that $l_{1}+\cdots +l_{n}=n+(m-1)$, 
and vertex algebra elements 
$v_{1}, \dots, v_{n+(m-1)}$ 
are included in summation for indexes for ${C}_{m}^{n}(V, \W, \U, \F)$.  
The conditions for the domains of absolute convergency for $\mathcal M$, i.e.,  
$|c_{l_{1}+\cdots +l_{i-1}+p}-\zeta_{i}| 
+ |c_{l_{1}+\cdots +l_{j-1}+q}-\zeta_{i}|< |\zeta_{i}
-\zeta_{j}|$, 
for $i$, $j=1, \dots, k$, $i\ne j$, and for $p=1, \dots,  l_i$ and $q=1, \dots, l_j$, 
 for the series \eqref{Inm}  
are more restrictive then for $m-1$. 
The conditions for $\mathcal I^n_{m-1}(\Phi)$ to be extended analytically  
to a rational function in $(c_1(p_{1})$,
 $\dots$, $c_{n+(m-1)} (p_{n+(m-1)}))$,
 with positive integers $N^n_{m-1}(v_{i}, v_{j})$,
depending only on $v_{i}$ and $v_{j}$ for $i, j=1, \dots, k$, $i\ne j$, 
are included in the conditions for $\mathcal I^n_m(\Phi)$. 

Similarly, the second condition for \eqref{Jnm}, 
of is absolute convergency
and analytical extension to a
rational function 
in $(c_{1}(p_1), \dots, c_{m+n}(p_{m+n}))$, with the only possible poles at
$c_{i}(p_i)=c_{j}(p_j)$, 
of orders less than or equal to 
$N^n_m(v_{i}, v_{j})$, for $i, j=1, \dots, k$, $i\ne j$,  
for \eqref{Jnm}  
when 
$c_{i}(p_i)\ne c_{j}(p_j)$, $i\ne j$
$|c_{i}(p_i)| > |c_{k}(p_k)|>0$ for
 $i=1, \dots, m$, 
and 
$k=m+1, \dots, m+n$
includes the same condition for $\mathcal J^n_{m-1}(\Phi)$. 
Thus we obtain the conclusion of Lemma. 
\end{proof}


\begin{thebibliography}{99}


\bibitem{Knizhka}
 Di Francesco, P., Mathieu, P., S\'en\'echal, D.:
Conformal Field Theory. Graduate Texts in Contemporary Physics,
Springer 1996


\bibitem
{BG} Ya. V. Bazaikin, A. S. Galaev. Losik classes for codimension one foliations,  
Mathematics of Jussieu (2021) doi:10.1017/S1474748020000596. 

\bibitem
{BGG} Ya. V. Bazaikin, A. S. Galaev, and P. Gumenyuk. 
Non-diffeomorphic Reeb foliations and modified Godbillon-Vey class, arXiv:1912.01267.
Math. Z. (2021).  
https://doi.org/10.1007/s00209-021-02828-1 

\bibitem
{BZF} {D. Ben-Zvi and E. Frenkel} {\it Vertex algebras on algebraic curves}.
American Mathematical Society, 2 edition, 2004.    


\bibitem
{Bott} R. Bott, {\it Lectures on characteristic classes and foliations.}
Springer LNM 279 (1972), 1--94.

\bibitem
{BH} R. Bott and A. Haefliger, On characteristic classes of $\Gamma$-foliations,
Bull. Amer. Math. Soc. 78 (1972), 1039--1044.

\bibitem
{BS} {R. Bott, G. Segal} 
The cohomology of the vector fields on a manifold. Topology. 
V. 16, Issue 4, 1977, Pages 285--298. 

\bibitem
{CM} M. Crainic and I. Moerdijk, 
 ${\rm \check C}$ech-De Rham theory for leaf spaces of foliations. Math.
Ann. 328 (2004), no. 1--2, 59--85.


\bibitem
{DNF} 
 B. A. Dubrovin, A. T. Fomenko, and S. P. and Novikov.  {\it  Modern geometry: methods and applications.}  
 Graduate Texts in Mathematics, 93. Springer-Verlag, New York, 1992. 

\bibitem
{FHL} I. Frenkel, Y. Huang and J. Lepowsky, On axiomatic
approaches to vertex operator algebras and modules, {\it Memoirs
American Math. Soc.} {\bf 104}, 1993.

\bibitem
{FMS} Ph. Francesco, P. Mathieu, and D. Senechal. Conformal Field Theory. 
 Graduate Texts in Contemporary Physics. 1997. 

\bibitem
{Fuks1}
D. B. Fuks, {\it Cohomology of infinite-dimensional Lie algebras,}  
Contemporary Soviet Mathematics, Consultunt Bureau, New York, 1986.


\bibitem
{Fuks2} D. B. Fuchs, Characteristic classes of foliations. Russian
Math. Surveys, 28 (1973), no. 2, 1--16. 

\bibitem
{Galaev} A. S. Galaev Comparison of approaches to characteristic classes of foliations, 
arXiv:1709.05888. 

\bibitem{G} E. Ghys
L'invariant de Godbillon-Vey. Seminaire Bourbaki, 41--eme annee, n 706, S. M. F. Asterisque 177--178 (1989).  

\bibitem{GF1}{ I. M. Gelfand and D. B. Fuks}, Cohomologies of the Lie algebra of vector fields on  the circle,  Funktional. Anal, i Prilozen. 2 (1968), no. 3, 32-52; ibid. 4 (1970), 23-32.


\bibitem{GF2}{ I. M. Gelfand and D. B. Fuks}, Cohomologies of the Lie algebra of tangent vector fields of a smooth manifold. I, II, Funktional. Anal, i Prilozen. 3 (1969), no. 3, 32-52; ibid. 4 (1970), 23-32.

\bibitem{GF3}{I. M. Gel'fand and D. B. Fuks}, Cohomology of the Lie algebra of formal vector fields, Izv. Akad. Nauk SSSR Ser. Mat. 34 (1970), 322-337, Math. USSR-Izv. 4 (1970), 327-342.



\bibitem
{Gu1} R. C. Gunning. {\it Lectures on Riemann surfaces },
Princeton Univ. Press, (Princeton, 1966).

\bibitem
{Gu2} R. C. Gunning. {\it Lectures on Vector Bundles Over Riemann Surfaces.} 
 Princeton University Press, 1967


\bibitem
{H} Y.-Zh. Huang, A cohomology theory of grading-restricted vertex algebras,
Comm. Math. Phys. 327 (2014), no. 1, 279--307. 

\bibitem
{Huang0} { Y.-Zh. Huang } Differential equations and conformal field theories. 
Nonlinear evolution equations and dynamical systems, 61--71, World Sci. Publ., River Edge, NJ, 2003.



\bibitem
{Huang} {Y.-Zh.  Huang} 
A cohomology theory of grading-restricted vertex algebras. 
Comm. Math. Phys. 327 (2014), no. 1, 279--307. 


\bibitem
{Hu3} Y.-Zh. Huang, The first and second cohomologies 
of grading-restricted vertex algebras, 
 Communications in Mathematical Physics volume 327, 261--278 (2014)

\bibitem
{H2}
Y.-Zh. Huang, {\it Two-dimensional conformal geometry and vertex operator
algebras}, Progress in Mathematics, Vol. 148,
Birkh\"{a}user, Boston, 1997.

\bibitem
{IZ} P. Iglesias-Zemmour. {\it Diffeology}, Mathematical Surveys and Monographs
Volume: 185; 2013; 439 pp. 

\bibitem
{K} {V. Kac}:   
 {\it Vertex Operator Algebras for Beginners}, 
University Lecture Series \textbf{10}, AMS, Providence 1998.

\bibitem
{Ko} { D. Kotschick}:  
Godbillon-Vey invariants for families of foliations. 
Symplectic and contact topology: interactions and perspectives (Toronto, ON/Montreal, QC, 2001), 
131--144, Fields Inst. Commun., 35, Amer. Math. Soc., Providence, RI, 2003. 


\bibitem{KN1} 
Krichever I.M., S. P. Novikov, 
Algebras of Virasoro type, Riemann surfaces and the structures of soliton theory (Russian),
 Funktsional. Anal, i Prilozhen. 21 (1987), no. 2, 46-63. 
\bibitem{KN2} 
Krichever I.M., S. P. Novikov, 
 Algebras of Virasoro type, Riemann surfaces and strings in Minkowski space (Russian),
 Funktsional. Anal, i Prilozhen. 21 (1987), no. 4, 47-61, 96, 

\bibitem{KN3} 
Krichever I.M., S. P. Novikov, Algebras of Virasoro type, the energy-momentum tensor, and operator expansions 
on Riemann surfaces (Russian), Funktsional. Anal, i Prilozhen. 23 (1989), no. 1, 24-40, 
English transl., Funct. Anal. Appl. 23 (1989), no. 1, 19-33. 



\bibitem
{L} {S. Lang}.  {\it Elliptic functions}. With an appendix by J. Tate. Second edition. Graduate Texts in
Mathematics, 112. New York: Springer-Verlag, 1987



\bibitem
{LosikArxiv} {\sc  Losik, M. V.}: Orbit spaces and leaf spaces of foliations as generalized manifolds, 
  arXiv:1501.04993. 



\bibitem{Lawson} H. B. Lawson,  Foliations. Bull. Amer. Math. Soc. 80 (1974), no. 3, 369--418. 

\bibitem
{Li} { J. I. Liberati}: Cohomology of vertex algebras. J. Algebra 472 (2017), 259--272.


\bibitem
{Losik} {M. V. Losik}. Orbit spaces and leaf spaces of foliations as generalized manifolds, 
  arXiv:1501.04993. 


\bibitem
{MSV}
F. Malikov, V. Schechtman, A. Vaintrob, 
Chiral de Rham complex.
Comm. Math. Phys. 204 (1999), no. 2, 439--473.


\bibitem
{N0} S. P.  Novikov. The topology of foliations. (Russian) Trudy Moskov. Mat. Obsch. 14 1965 248--278.

\bibitem
{N1} S. P. Novikov. Topology of generic Hamiltonian foliations on Riemann surfaces. 
Mosc. Math. J. 5 (2005), no. 3, 633--667, 743.

\bibitem
{N2}  
S. P. Novikov. Topology of Foliations given by the real part of holomor-
phic one-form, preprint, 2005, arXive math.GT/0501338. 


\bibitem{vinogradov} {A. M. Vinogradov.}  
{\it Cohomological analysis of partial differential equations and secondary calculus.}  
 Translations of Mathematical Monographs, 204. 
American Mathematical Society, Providence, RI, 2001. xvi+247. 


\bibitem
 {FQ} { F. Qi},     
 Representation theory
and cohomology theory of meromorphic open string vertex algebras, Ph.D. dissertation, (2018).  

\bibitem
 {RS} 
Razumov, A. V.; Saveliev, M. V. Lie groups, differential geometry, and nonlinear integrable systems. 
Nonassociative algebra and 
its applications (Sao Paulo, 1998), 32--336, Lecture Notes in Pure and Appl. Math., 211, Dekker, New York, 2000.

\bibitem
{Schl0} Schlichenmaier, M. 
Krichever-Novikov type algebras. An introduction.
 Lie algebras, Lie superalgebras, vertex algebras and related topics, 181-220,
Proc. Sympos. Pure Math., 92, Amer. Math. Soc., Providence, RI, 2016.

\bibitem
{Schl1} 
 Schlichenmaier, M., Krichever-Novikov algebras for more than two points, Lett. Math. Phys.
19(1990), 15-165.

\bibitem
{Schl2} 
Krichever-Novikov algebras for more than two points: explicit generators, Lett. Math.
Phys. 19(1990), 327-336.


\bibitem
{S1} G. Segal, 
The definition of conformal field theory,
in: {\em Differential geometrical methods in theoretical physics 
(Como, 1987)}, 
NATO Adv. Sci. Inst. Ser. C Math. Phys. Sci., 250, 
Kluwer Acad. Publ., Dordrecht, 1988, 165--171. 


\bibitem
{S3}
G. B.~Segal, Two-dimensional conformal field theories and modular
functors, in: {\em Proceedings of the IXth International Congress on
Mathematical Physics, Swansea, 1988},
Hilger, Bristol, 1989, 22--37.



\bibitem{Thur} W. Thurston, Non-cobordant foliations on $S^3$. Bulletin Amer. Math. Soc. 78 (1972), 511--514.

\bibitem
{TUY} A. Tsuchiya, K. Ueno, and Y. Yamada, Y.: Conformal field
theory on universal family of stable curves with gauge symmetries,
Adv. Stud. Pure. Math. \textbf{19} (1989), 459--566.


\bibitem
{W} C. Weibel. {\it An introduction to homological algebras}, {\it Cambridge Studies in Adv. Math.},
Vol. 38, Cambridge University Press, Cambridge, 1994.


\bibitem
{Y} {A. Yamada, A.}. Precise variational formulas for abelian
differentials. Kodai Math.J. \textbf{3} (1980) 114--143.


\bibitem
{Zhu} Y. Zhu. Modular invariance of characters of vertex operator algebras. 
J. Amer. Math. Soc. 9(1), 237-302

\bibitem{vinogradov} {\sc Vinogradov, A. M.}  
Cohomological analysis of partial differential equations and secondary calculus. 
 Translations of Mathematical Monographs, 204. 
American Mathematical Society, Providence, RI, 2001. xvi+247. 

\end{thebibliography}
\end{document}